\documentclass[11pt,a4paper]{article}
\usepackage{geometry}

\UseRawInputEncoding
\usepackage[utf8]{inputenc}
\usepackage[T1]{fontenc}
\usepackage{amsmath}
\usepackage{amsthm}
\usepackage{amssymb}
\usepackage{graphicx}
\usepackage{caption}
\usepackage{multirow}
\usepackage{mathrsfs}
\usepackage{tabularx}
\usepackage{dsfont}
\usepackage{hyperref}
\usepackage[dvipsnames]{xcolor}
\usepackage{verbatim}
\usepackage[shortlabels]{enumitem}
\hypersetup{colorlinks=true, linkcolor=blue!60!black, urlcolor=red!60!black, citecolor=blue!60!black}
\usepackage{setspace}
\usepackage[ruled,vlined]{algorithm2e}
\usepackage{subcaption}
\usepackage{booktabs}

\SetKwProg{Init}{Initialisation}{}{}
\usepackage{float}
\usepackage[normalem]{ulem}
\usepackage{array,multirow,makecell}
\usepackage[justification=centering]{caption}
\setcellgapes{1pt}
\makegapedcells
\newcolumntype{R}[1]{>{\raggedleft\arraybackslash }b{#1}}
\newcolumntype{L}[1]{>{\raggedright\arraybackslash }b{#1}}
\newcolumntype{C}[1]{>{\centering\arraybackslash }b{#1}}
\makeatletter
\newcommand\justify{%
  \let\\\@centercr
  \rightskip\z@skip
  \leftskip\z@skip}
\makeatother

\newtheorem{thm}{Theorem}[section]
\newtheorem{coro}[thm]{Corollary}
\newtheorem{prop}[thm]{Proposition}
\newtheorem{lem}[thm]{Lemma}
\newtheorem{ass}{Assumption}[section]

\theoremstyle{definition}
\newtheorem{defi}{Definition}[section]
\newtheorem{exe}[thm]{Example}
\newtheorem{rmq}{Remark}[section]

\newenvironment{preuve}{\begin{proof} \rm}{\end{proof}}

\newtheorem*{prop*}{Proposition}
\newtheorem*{rmq*}{Remark}

\numberwithin{equation}{section}

\newcommand{\p}{\mathbb{P}}
\newcommand{\R}{\mathbb{R}}

\newcommand{\e}{\mathbb{E}}
\newcommand{\g}{\mathbb{G}}

\newcommand{\ind}{\mathds{1}}
\newcommand{\f}{\mathbb{F}}

\newcommand{\Rr}{\mathcal{R}}
\newcommand{\U}{\mathcal{U}}
\newcommand{\I}{\mathcal{I}}

\newcommand{\M}{\mathcal{M}}
\newcommand{\Gr}{\mathcal{G}}

\newcommand{\N}{\mathcal{N}}
\newcommand{\Qr}{\mathcal{Q}}
\newcommand{\Y}{\mathcal{Y}}
\newcommand{\Sr}{\mathcal{S}}
\newcommand{\Z}{\mathcal{Z}}

\newcommand{\E}{\mathcal{E}}
\newcommand{\K}{\mathcal{K}}

\newcommand{\F}{\mathcal{F}}

\newcommand{\A}{\mathcal{A}}
\newcommand{\Lr}{\mathcal{L}}

\newcommand{\ABS}[1]{{\left| #1 \right|}} 
\newcommand{\PAR}[1]{{\left(#1\right)}} 
\newcommand{\SBRA}[1]{{\left[#1\right]}} 
\newcommand{\BRA}[1]{{\left\{#1\right\}}} 
\newcommand{\NRM}[1]{{\left\Vert #1\right\Vert}} 

\DeclareMathOperator{\Tr}{Tr}

\DeclareMathOperator{\Proj}{Proj}
\DeclareMathOperator{\argmin}{argmin}
\DeclareMathOperator{\dist}{dist}
\DeclareMathOperator{\diag}{diag}

\DeclareMathOperator{\Err}{Err}

\DeclareMathOperator{\loc}{loc}

\newcommand\blfootnote[1]{%
  \begingroup
  \renewcommand\thefootnote{}\footnote{#1}%
  \addtocounter{footnote}{-1}%
  \endgroup
}

\title{
Deep numerical schemes for systems of Ergodic BSDEs with applications to regime-switching forward utilities \blfootnote{Acknowledgements:  The authors research is part of the ANR project DREAMeS (ANR-21-CE46-0002) and benefited from the support of respectively the "Chair Risques Emergents en Assurance"  and  "Chair Impact de la Transition Climatique en Assurance"  under the aegis of Fondation du Risque, a joint initiative by  Risk and Insurance Institute of  Le Mans,  and MMA-Cov\'ea and Groupama respectively. }}

\author{Guillaume Broux-Quemerais\textsuperscript{1} \and Sarah Kaakai\textsuperscript{2} \and Anis Matoussi\textsuperscript{1} \and Wissal Sabbagh\textsuperscript{1}}

\date{}

\begin{document}

\maketitle

\rightskip=1.5cm
\leftskip=1.5cm

\footnotetext[1]{\small Laboratoire Manceau de Math\'ematiques \& FR CNRS N\textsuperscript{o} 2962, Institut du Risque et de l'Assurance, Le Mans Université}
\footnotetext[2]{Laboratoire Analyse, Géométrie et Applications, UMR CNRS 7539, Université Sorbonne Paris Nord}

\begin{small}
\begin{center}
\section*{Abstract}
\end{center}

In this paper, we introduce two neural-network-based numerical schemes for solving systems of coupled ergodic Backward Stochastic Differential Equations (eBSDEs), motivated by the approximation of optimal strategies within the framework of forward utilities in a regime-switching stochastic factor model. Our approach builds on the representation of such models through systems of eBSDEs introduced in \cite{hu2020systems}. We first establish a link between the solution of the system of ergodic BSDEs and that of an associated multidimensional BSDE with random terminal time, given by the hitting time of the positive recurrent stochastic factor. Building on this representation, we introduce a locally additive deep learning scheme obtained by minimizing aggregated local error terms. We then present a  new Deep Galerkin Method (DGM) inspired algorithm that minimizes the residual of the associated ergodic PDE system, relying on a representation of the ergodic cost. Finally, we apply this framework to regime-switching forward utilities in a stochastic factor model. We first derive a  general consistency SPDE that characterizes regime-switching forward utilities and retrieve their representation with systems of ergodic BSDEs in the homothetic case.  Numerical experiments demonstrate the performance of the proposed methods, with a particular focus on the impact on forward preferences of taking into account regime switches.
\end{small}
\newpage

\rightskip=0cm
\leftskip=0cm

\vspace{1cm}
\section*{Introduction}
In this paper, we are interested in the numerical approximation of certain classes of forward performance/utility processes in a regime-switching financial market, and their associated optimal decision criterion. Introduced by \cite{musiela2006investments}, forward utilities offer an interesting alternative to the classical setting of expected utility maximization at a terminal time. This forward-looking approach enables the dynamic adjustment of the decision criteria, starting from preferences which are known at an initial time, rather than imposing a potentially distant and arbitrary time horizon. In a continuous setting, the agent's preferences are represented by a utility random field $U(t, \cdot)$. This preference criterion maintains time consistency within the given investment or decision-making context, in the sens that  if $X_t^\pi$ is the observable process (typically the wealth), resulting from the admissible strategy  $\pi$, then the preference process $U(t, X_t^\pi)$ is a supermartingale, and there exists an optimal strategy such that the preference process is a martingale.

Since their introduction, there have been significant theoretical advancements in the field of forward utilities. In a general framework, \cite{musiela2010stochastic} established a sufficient condition for time-consistency when the utility random field follows dynamics of It\^o type. This condition takes the form of a nonlinear SPDE of Hamilton-Jacobi-Bellman (HJB) type satisfied by $U$. By establishing a correspondence between the SPDE and the compound of two SDEs, the authors in \cite{nicole2013exact} provide a method to construct forward utilities. This work has been extended to consistent utility of investment and consumption  in \cite{el2018consistent}, and has been applied, for instance,  to derive consistent utilities in stochastic factor market models (see e.g. \cite{nadtochiy2014class}, \cite{avanesyan2020construction}, \cite{liang2017representation}). To the best of our knowledge, the study of forward utilities in regime-switching markets has only been tackled in \cite{hu2020systems}. Forward utilities have found diverse applications over recent years, including but not limited to option valuation, insurance,  mean field games (\cite{lacker2019mean}, \cite{dos2021forward}), long term interest rate modeling (\cite{el2022ramsey}), risk measures  (\cite{chong2019pricing} or more recently pension design (\cite{hillairet2022time},\cite{ng2024optimal}). Surprisingly, the subject of numerical methods for forward preferences remains largely unexplored,  despite its critical importance for practical applications.  In \cite{gobet2018convergence}, a general approach is proposed using  strong approximations of compounds of random maps. In this paper, we focus on a different point of view, introduced in \cite{hu2020systems}, to develop new numerical methods for forward utilities in regime-switching markets, based on a representation with means of systems of ergodic BSDEs.

Regime-switching models, also called Markov-modulated dynamics, are designed to capture financial markets that exhibit multiple states. In addition to the standard Brownian motion, which drives stock price dynamics, these models incorporate an additional source of randomness represented by a finite-state continuous-time Markov chain, where each state corresponds to a distinct market regime. For instance, a two-regime model may describe a "bull market," characterized by rising asset prices, and a "bear market," marked by declining prices. In this framework, the Brownian motion captures persistent microeconomic effects, while the Markov chain accounts for occasional, macroeconomic shifts in market conditions. Regime-switching models have been extensively studied in the context of option pricing; see, for example, \cite{buffington2002american}, \cite{Guo2004}, \cite{Yao2006}, \cite{jobert2006option}. Portfolio optimization problems within regime-switching frameworks have also been considered, including mean-variance portfolio selection in \cite{zhou2003markowitz} and utility maximization across various models, see \cite{zariphopoulou1992investment}, \cite{bauerle2004portfolio}, \cite{sotomayor2009explicit}, and \cite{fu2014portfolio}.

Our motivation is the representation of an agent's preferences in an incomplete financial market presenting multiple states. Formally, the regime transitions are governed by a continuous-time Markov chain $\alpha$, with a finite state space $\I = \BRA{1, \dots, I}.$ In this context, the decision criteria are built upon a family of utility random fields $\PAR{U^{i}}_{i \in \I}$, each representing the agent's preferences in regime $i$. The regime-switching utility is the càdlàg process defined as
\begin{eqnarray*}
U(t, x) = U^{\alpha_t}(t,x) = \sum_{i \in \I} U^i(t,x) \ind_{\{\alpha_t = i\}}.
\end{eqnarray*}
We focus on a stochastic factor model in which the dynamics of stock prices are influenced by a stochastic factor $V$. This leads to the study of regime-switching preferences based on families of homothetic utilities, which are expressed as separable functions of the stochastic factor. For each $i \in \I$, these are given by $U^{i}(t,x) = u(x)e^{f^{i}(t,V_t)}$, where $u$ is a standard exponential or power utility function (with an additive expression in the logarithmic case). While the main focus of \cite{hu2020systems} is on existence and uniqueness of the Markovian solution to systems of eBSDEs, they also characterize the
functions $f^{i}$ as solutions of such systems for homothetic utilities in power form.

Systems of ergodic BSDEs generalize the notion of ergodic BSDEs introduced in \cite{fuhrman2009ergodic} (see also \cite{buckdahn1999ergodic} with a different formulation with stationary processes ), \cite{debussche2011ergodic}, \cite{hu2019ergodic}), to the case of multidimensional and coupled components $Y$. Informally, a system of ergodic BSDEs with generators $(F^{i})_{1 \leq i \leq I}$ and coupling terms $(G^{i})_{1 \leq i \leq I}$ is a system of backward stochastic differential equation on an infinite horizon, whose solution is a triplet $((Y^{i}, Z^{i})_{1 \leq i \leq I}, \lambda)$ where $\lambda \in \R$ and for all $i = 1, ..., I$, processes $Y^{i}, Z^{i}$ are adapted and satisfy for any $T>0$, $\p$-a.s for any $0 \leq t \leq T$:
\begin{eqnarray} \label{eq:sebsdeintroarticle}
Y_{t}^{i} &=& Y_{T}^{i} + \int_{t}^{T} \PAR{F^{i}(V_{s}, Z_{s}^{i}) + G^{i}(Y_{s})}ds - \lambda (T-t) - \int_{t}^{T} (Z_{s}^{i})^{\top} dW_{s}.
\end{eqnarray}
Existence and uniqueness of a Markovian solution to systems of eBSDEs of the form \eqref{eq:sebsdeintroarticle} are established in \cite{hu2020systems}, under assumptions analogous to those for classical ergodic BSDEs: locally Lipschitz
conditions on $F^{i}$ and strong dissipativity of the stochastic factor. Such systems are
genuinely multidimensional quadratic BSDEs, for which existence and uniqueness have been extensively studied, see \cite{frei2011financial} \cite{Tevzadzequadr08, Frei2014, Hu2016,
Xingquadr16, harterrichou19}. In \cite{hu2020systems}, the quadratic growth is handled by a truncation argument exploiting a specific property of the solution, namely the
boundedness of the $Z$ component.

Numerical methods for ergodic BSDEs remain, by contrast, largely unexplored. The authors
of \cite{gobet2024numerical} propose a fully implementable scheme with an approximation-error analysis based on nonlinear Feynman--Kac formulas. In a different
direction, machine-learning methods for high-dimensional nonlinear PDEs and BSDEs have
been investigated extensively in recent years \cite{han2017deep, hure2020deep, germain2021neural,
kapllani2020deep}, and were first extended to the ergodic setting in \cite{deepbroux2024}.
We pursue this line for systems of ergodic BSDEs arising from regime-switching forward
utilities.

\paragraph{Objective and contributions.}
The aim of this paper is to develop numerical methods for the approximation of Markovian solutions to systems of ergodic BSDEs, with a particular focus on those arising from homothetic regime-switching forward utilities.

Our first contribution is theoretical. We state rigorously the equivalence between the system of ergoidc BSDEs, the ergodic PDE system and the ergodic BSDE with jumps. These correspondences underpin the foundations of our algorithms. Computing the dynamics of $U(t, X_t^\pi)$ by means of a generalized Itô--Ventzell formula with jumps \cite{oksendal2007ito, matoussi2022dynamic}, we derive a consistency SPDE with jumps (Theorem
\ref{thm:consistswitch}) that characterizes regime-switching forward utilities in our framework, providing a general consistency condition in regime-switching models. Applied to homothetic preferences, it identifies the systems of ergodic BSDEs associated with power, exponential and logarithmic utilities, recovering in particular the
power-utility system of \cite{hu2020systems}.

Our second contribution is numerical. We introduce two neural-network schemes for the simultaneous approximation of the ergodic cost $\lambda$ and the processes $Y$ and $Z$: a locally additive deep BSDE scheme (LAeBSDE) and an extension of the Deep Galerkin Method (DGM). The probabilistic foundation of the first scheme rests on a random-horizon
reformulation of the ergodic system. Since the initial agent's utility is known, an initial condition $Y_0^{\alpha_0} \in \R$ is naturally available. Leveraging the recurrence of the stochastic factor $V$, we introduce the first return time $\tau$ of $V$ to its initial point $v_0$, which is almost surely finite under the dissipativity
assumption and satisfies $Y_\tau = Y_0$. We show that
the solution of the ergodic system \eqref{eq:sebsdeintroarticle} then coincides with that of an associated system of BSDEs with random terminal time $\tau$, and, under suitable
integrability of $\tau$, that this solution is unique, building on \cite{pardoux1998backward}
and on the fact that $\lambda$ is pinned down by the bridge condition $Y_\tau = Y_0$. The argument is more involved than in \cite{deepbroux2024} since the normalization only fixes one point $v_{0}$ in one state $i_{0}$ of the solution.

Building on this reformulation, the LAeBSDE scheme discretizes $V$ by an Euler scheme, estimates the horizon $\tau$ along each trajectory, and treats $\lambda$ as a trainable parameter, optimized jointly with two networks approximating $Y$ and $Z$ through an aggregation of local losses. The DGM scheme follows a distinct route by approximating the Markovian solution directly through the minimization of the residual of the associated ergodic PDE system at randomly sampled points, the ergodic cost being recovered from its invariant-measure representation.

Finally, we assess the performance of the two algorithms on several examples. We first construct a class of examples admitting an explicit solution, by starting from a sublinear ansatz and choosing the market price of risk accordingly. We then study a two-state market model with regime-switching power utility and an affine market price of risk, where a comparison with the non-switching case highlights the economic implications of our framework for
forward preferences.

\vspace{0.3cm}
\noindent
The paper is organized as follows. In Section \ref{sect:coupledsebsde}, we recall the framework and the existence and uniqueness results associated with the system of ergodic BSDEs. We also state the ergodic PDE and BSDE with jumps associated with the system of ergodic BSDE. In Section \ref{section:algoswitch}, we establish a correspondence between systems of eBSDEs and BSDEs with random terminal time, which is needed for the definition of the locally additive scheme. We then extend the Deep-Galerkin method to the approximation of ergodic PDEs. Section \ref{section:forwardutilitiesswitch} introduce the regime-switching stochastic factor market model and the notion of forward utilities. We then derive the consistency SPDE that allows the identification of systems of ergodic BSDEs associated with power, exponential and logarithmic homothetic utilities. Finally, in Section \ref{section:numresswitch} we present numerical results on one benchmark example and interpret financial implications of regime-switches through an example of power utility with affine market price of risk.

\newpage

\noindent\textbf{Notations:}\\
All stochastic processes in the sequel are defined on a standard probability space $\PAR{\Omega, \F,\f,  \p}$, on which we consider a $d$-dimensional Brownian motion $W$, and denote $\f = \PAR{\F_{t}}_{t \geq 0}$ its augmented filtration. 

For $x \in \R^{d}$, we denote $x^{\top}$ the transpose of vector $x$, $\NRM{.}$ the usual norm $\NRM{x} = \Tr(x x^{\top})^{\frac{1}{2}}$, $\dist(x, \Pi)$ the distance function of $x$ to a closed convex subset $\Pi \subset \R^{d}$ and $\diag(x) = \PAR{\diag(x)_{i j}}_{1 \leq i, j \leq d}\in \R^{d \times d}$, the matrix such that $\diag(x)_{ii} = x^{i}$ and $\diag(x)_{i j} = 0$ for $i \neq j$. For $X = \PAR{X_{t}}_{t \geq 0}$ a càdlàg stochastic process, we denote $X_{t^-}$ the left-limit of this process at time $t$. 

For any integer $k \geq 1$ and positive constant $B \in \R^{+*}$, let $\varphi_{B}$  denote the projection on the centered ball of $\R^{k}$ of radius $B$. We denote $L^{2}$ the space of square integrable random variables and also introduce the usual spaces of solution for $\gamma \in \R$ and $\tau$ a $\f$ stopping time:
\begin{small}
\begin{eqnarray*}
\Sr^{2}(\gamma, \tau) &=& \BRA{(\varphi_{t})_{t \geq 0}, \, \R^{k}\text{-valued progressively measurable process} \, \text{s.t.} \, \e \SBRA{\underset{0 \leq s \leq \tau}{\sup} e^{\gamma s} \NRM{\varphi_{s}}^{2}} < \infty.} \\
\M(\gamma, \tau) &=& \BRA{(\varphi_{t})_{t \geq 0}, \, \R^{k}\text{-valued progressively measurable process} \, \text{s.t.} \, \e \SBRA{\int_{0}^{\tau} e^{\gamma s} \NRM{\varphi_{s}}^{2}} < \infty.}
\end{eqnarray*}
\end{small}

\section{System of coupled ergodic BSDEs} \label{sect:coupledsebsde}

We are interested in the numerical study of coupled systems of ergodic BSDEs, as introduced in \cite{hu2020systems}. Let $I > 1$, and denote $\I = \BRA{1, \dots, I}$. In the following, we call systems of ergodic BSDEs with generators $(F^{i})_{i \in \I}$ and coupling terms $(G^{i})_{i \in \I}$, a multidimensional BSDE on an infinite horizon, whose solution is a triplet $((Y^{i}, Z^{i})_{i \in \I}, \lambda)$ where $\lambda \in \R$ and $(Y^{i}, Z^{i})_{1\le i \le I}$ is a family of  $\f$-adapted  processes taking values in $\R \times \R^d$, satisfying for any $T>0$ and $\p$-a.s for any $0 \leq t \leq T$:
\begin{eqnarray} \label{eq:sebsdeform}
Y_{t}^{i} &=& Y_{T}^{i} + \int_{t}^{T} \PAR{F^{i}(V_{s}^{v_0}, Z_{s}^{i}) + G^{i}(Y_{s})}ds - \lambda (T-t) - \int_{t}^{T} (Z_{s}^{i})^{\top} dW_{s},
\end{eqnarray}
 where  $V^{v_0}$ is a stochastic factor solution of 
\begin{equation}\label{Vequation}
    dV_t^{v_0} = \mu(V_{t}^{v_0})dt + \kappa^{\top} dW_{t}, \quad V_{0}^{v_0} = v_{0}\in \R^{d'} , \quad \kappa \in \R^{d \times d'}, 
\end{equation}
such that 
\begin{ass} \label{ass:weakdissass}  $\mu$ is Lipschitz, and there exists a constant $C_{\mu} > 0$ such that for any $v, \bar{v} \in \, \R^{d'}$
            \begin{eqnarray} \label{disscond:switch}
            \PAR{\mu(v) - \mu(\bar{v})}^{\top}(v - \bar{v}) \leq - C_{\mu}\NRM{v - \bar{v}}^{2}.
            \end{eqnarray}
and $\kappa \kappa^{\top}$ is a positive definite matrix. In particular, the diffusion $V^{v_0}$ is exponentially ergodic.
\end{ass}

The next lemma provides an estimate on the running supremum of the stochastic factor $V^{v_0}$, which will be used in Section \ref{sect:connectionrdtbsde} to establish integrability properties of the Markovian solution to the system of ergodic BSDEs. The proof is postponed to Appendix \ref{proofLemma 1.1}.
\begin{lem} \label{lem:runningsupV}
Suppose that Assumption \ref{ass:weakdissass} holds and let $V^{v_0}$ be the solution of \eqref{Vequation}. Then for every $q \geq 1$, there exists a constant $C_{q} > 0$, depending only on $q$, $C_{\mu}$ and $\kappa$, such that for every stopping time $\tau$,
\begin{eqnarray}
\e \SBRA{\underset{0 \leq t \leq \tau}{\sup} \ABS{V_{t}^{v_0}}^{2q}} \leq C_{q} \PAR{1 + \ABS{v_{0}}^{2q} + \e \SBRA{\tau^{q}}}.
\end{eqnarray}
\end{lem}

\subsection{Existence and uniqueness of Markovian solution}

We consider an extension of the setting studied in \cite{hu2020systems}, allowing for a more general coupling term $(G^i)_{1 \leq i \leq I}$, defined in \eqref{eq:Gexprsum}, instead of the specific coupling function $g(y)=e^{y}-1$ considered therein. Under Assumption \ref{ass:FGgrowthass} (iii), the proof follows essentially the same arguments as in Theorem 3.2 of \cite{hu2020systems}, and the corresponding existence and uniqueness result for bounded Markovian solutions to \eqref{eq:sebsdeform} remains valid. We recall this result below under the additional assumptions imposed in our framework.
\begin{ass} \label{ass:FGgrowthass}
The following properties hold:
\begin{itemize}
\item [i)]There exists a positive constant $K$ such that $\forall v \in \R^{d'}$, $\ABS{F(v, 0)} \leq K$.
\item [ii)]There exists positive constants $C_{v}$ and $C_{z}$ such that for $i\in\I$, $\forall v, \bar{v} \, \in \R^{d'}$, $\forall z, \bar{z} \, \in \R^{d\times I}$
\begin{eqnarray}
\ABS{F^{i}(v, z) - F^{i}(\bar{v}, z)} &\leq& C_{v} \PAR{1 + \NRM{z}} \NRM{v-\bar{v}}, \label{Ftroncv:switch} \\
\ABS{F^{i}(v, z) - F^{i}(v, \bar{z})} &\leq& C_{z} \PAR{1 + \NRM{z} + \NRM{\bar{z}}} \NRM{z - \bar{z}}. \label{Ftroncz:switch}
\end{eqnarray}
Moreover, we require that $C_{v} < C_{\mu}$ with $C_{\mu}$ as introduced  in \eqref{disscond:switch}.
\item [iii)] There exists an increasing $C^{1}$ function $g : \R \to \R$, such that, for all $i\in\I$ and $y \in \R^{I}$
\begin{eqnarray} \label{eq:Gexprsum}
& G^{i}(y)
=
\underset{j=1}{\overset{I}{\sum}} q^{ij} g\left(y^{j}-y^{i}\right), 
\end{eqnarray}
$g(\boldsymbol{0}) = 0$ and for all $y \in \R^{I}$, $g(y) - g(-y) \geq y$.
\end{itemize}
\end{ass}

Note that, under Assumption \ref{ass:FGgrowthass} (iii), the function $g$ is Locally Lipschitz on $\R$.
Hence, setting $L_{g}(M) = \underset{\ABS{x} \leq M}{\sup} g'(x)$ and $q^{*} = \underset{i \in \I}{\max} \underset{j \neq i}{\sum} q^{ij}$, the map $G : \R^{I} \to \R^{I}$ is Lipschitz on \begin{eqnarray*}
\mathcal{D}_{M} := \BRA{y \in \R^{I}, \, \ABS{y^{j} - y^{i}} \leq M, \, \forall \, i, j \in \I}.
\end{eqnarray*}
Indeed, for any $y,\bar y\in \mathcal D_M$ and any $i\in\I$,
\begin{eqnarray*}
\ABS{G^{i}(y) - G^{i}(\bar{y})} &=& \ABS{\sum_{j \neq i} q^{ij} \PAR{g(y^{j} - y^{i}) - g(\bar{y}^{j} - \bar{y}^{i})}} \\
&\leq& L_{g}(M) \sum_{j \neq i} q^{i} \ABS{(y^{j} - y^{i}) - (\bar{y}^{j} - \bar{y}^{i})} \\
&\leq& L_{g}(M) \sum_{j \neq i} q^{ij} \PAR{\ABS{y^{j} - \bar{y}^{j}} + \ABS{y^{i} - \bar{y}^{i}}} \\
&\leq& L_{g}(M) 2 \NRM{y - \bar{y}} \sum_{j \neq i} q^{ij} \\
&\leq& L_{g}(M) 2 \NRM{y - \bar{y}} q^{*}.
\end{eqnarray*}
and hence, for all $i \in \I$, $G^i$ is also locally Lipschitz on $\R$. 

In addition to the usual ergodicity of the stochastic factor $V^{v_0}$, ensured by Assumption \ref{ass:weakdissass}, the study of Markovian solutions to systems of ergodic BSDEs \eqref{eq:sebsdeform} also requires the irreducibility of the transition rate matrix associated with the coefficients $q^{ij}$.

\begin{ass} \label{ass:MC}
$\Qr= (q_{ij})_{1\leq i,j\leq I}$ is a  transition rate matrix, i.e. , ${q^{ii} = - \underset{j \neq i }{\sum} q^{ij}}$, for all   $i\in\I$. Furthermore,  we assume that for all $i \neq j$ $q^{ij} > 0$  and we denote 
$$q_{\min} = \underset{i \neq j}{\min} \, q^{ij}.$$
\end{ass}

\begin{thm} \label{thm:exsebsde}
Let $i_0 \in \I$, $v_0\in\R^{d'}$ and $ y_0 \in \mathbb{R}$. 
Under Assumptions \ref{ass:weakdissass}, \ref{ass:FGgrowthass}, \ref{ass:MC}, there exists a unique Markovian solution ${\PAR{(y^{i}(V_{t}^{v_0}),  z^{i}(V_{t}^{v_0}))_{i \in \I}, \lambda}_{t \geq 0}}$ to the system of ergodic BSDEs \eqref{eq:sebsdeform}, such that $y^{i_{0}}(v_{0}) = y_{0} \in \, \R$ and for all $i\in \I$, and $v\in \mathbb R^{d'}$, 
\begin{eqnarray}
\ABS{y^{i}(v)} &\leq& C \PAR{1 + \NRM{v}} \\
\ABS{z^{i}(v)} &\leq& \NRM{\kappa} \frac{C_{v}}{C_{\mu} - C_{v}} \\
\ABS{y^{i}(v) - y^{j}(v)} &\leq& \frac{1}{q_{\min}} \PAR{K + \frac{C_{v} C_{\mu} C_{z}}{(C_{\mu} - C_{v})^{2}}} := C_{Y}. \label{eq:unifboundy}
\end{eqnarray}
\end{thm}

\begin{rmq}
\label{rmq:uniqueness}
\begin{enumerate}
\item Notably, the constant $\lambda$ in the solution  of \eqref{eq:sebsdeform}, called the ergodic cost, is shared by all equations for all $i \in \I$, and do not depend on the choice of $i_0$, $v_0$ or $y_0$. 

\item A further important point is that uniqueness of Markovian solutions is ensured by fixing only one component, $y^{i_0}(v_0)$, of the vector-valued solution $y=(y^1,\ldots,y^I)$ at $i_{0}$, $v_0$. Hence, one cannot fix the whole vector $(y^1(v_0),\ldots,y^I(v_0))$ simultaneously. This constitutes a nontrivial obstacle for the numerical approximation of the solution.

\item The assumption of sublinearity of the solution components $y^{i}$ is required for uniqueness in Theorem \ref{thm:exsebsde}. In fact, \cite{hu2020systems} provide examples of solutions to ergodic BSDEs that do not satisfy this growth property.
\end{enumerate}
\end{rmq}

\paragraph{Link with semilinear PDE}

 The system of Makovian ergodic BSDEs \eqref{eq:sebsdeform}, represented by \\${\PAR{(y^{i}(V_{t}^{v_0}),  z^{i}(V_{t}^{v_0}))_{i \in I}, \lambda}_{t \geq 0}}$, gives a probabilistic representation of the following system of elliptic PDE, for all $i\in \I$
\begin{equation}\label{PDEequation}
    \mathcal L y^i(v)+ F^{i}(v, \nabla y^{i}(v)\kappa)+\sum_{j \in \I} 
q^{i j}g(y^{j}(v)-y^{i}(v))=\lambda
\end{equation}
where $\mathcal L$ denotes the generator of the semi-group associated to the SDE \eqref{Vequation}
\begin{prop}
Under Assumptions \ref{ass:weakdissass}, \ref{ass:FGgrowthass}, \ref{ass:MC}, $(y^{i})_{i \in \I}$ is a viscosity
solution of \eqref{PDEequation}.
\end{prop}

\begin{preuve}We consider, for all $\rho>0$, the following infinite horizon markovian BSDE system: for $t\geq 0$ and $i\in \I$ 
\begin{eqnarray} \label{eq:sinfbsdeform}
Y_{t}^{i,\rho,v} &=& Y_{T}^{i,\rho,v} + \int_{t}^{T} \PAR{F^{i}(V_{s}^v, Z_{s}^{i,\rho,v}) + \sum_{j \in \I} 
q^{i j}g(Y_{s}^{j,\rho,v}-Y_{s}^{i,\rho,v})-\rho Y_{s}^{i,\rho,v}}ds- \int_{t}^{T} (Z_{s}^{i,\rho,v})^{\top} dW_{s}.\nonumber\\
&&
\end{eqnarray}
and define for a fixed reference point, say $v_0 \in \mathbb R^{d'}$, $ y^{i,\rho}(v):=Y_{0}^{i,\rho,v}, \bar y^{i,\rho}(v):=Y_{0}^{i,\rho,v}-Y_{0}^{\alpha_0,\rho,v_0},$ for all $v\in\mathbb{R}^{d'}.$
By standard arguments, see e.g.\ the proof of Theorem~5.74 in~\cite{PR14}, 
$\bar{y}^{i,\rho}$ is a viscosity solution of the elliptic PDE
\begin{equation}\label{ellipticPDE}
\mathcal{L}y^i(v) + F^{i}\bigl(v,\nabla y^i(v)\kappa\bigr)+\sum_{j \in \I} 
q^{i j}g(y^{j}(v)-y^{i}(v))
= \rho y^i(v) + \rho y^{\alpha_0,\rho}(v_0)
\qquad v \in \mathbb{R}^{d'}.
\end{equation}
Following the arguments of \cite{hu2020systems}, we can prove that there exists a sequence $(\rho_n)_{n \in \mathbb{N}}$ such that
$\rho_n \searrow 0$, $\rho_n\, y^{\alpha_0,\rho_n}(v_0) \to \lambda$, and
$\bar{y}^{\,i,\rho_n}(v) \to u$ uniformly on $\mathbb{R}^{d'}$ as $n \to +\infty$.
Then, Remark~6.3 in~\cite{CIL92} implies that $y^i$ is a viscosity solution of~\eqref{PDEequation}.

\end{preuve}

\paragraph{Link with eBSDE with jumps} A useful representation of the solution of  \eqref{eq:sebsdeform} introduced  in \cite{hu2020systems} relies on the transformation of the system of ergodic BSDEs \eqref{eq:sebsdeform} into an ergodic BSDEs with jumps, where jumps occur at jump times of a continuous time Markov chain (CTMC) on $\I$, of transition rate matrix $\Qr=(q_{ij})_{1\leq i,j\leq I}$.\\
Assume that the probability space supports a random Poisson measure $N(dt, di, dj)$ on $\R^+ \times \I^2$, of intensity $q_{ij} dt\gamma(di)\gamma(dj)$, with $\gamma$ the counting measure on $ \I$. Let $\alpha$ be the CTMC of transition rate matrix $\Qr$, solution of
\begin{equation}
\label{eq:defalpha}
    \alpha_t  = \alpha_0 + \int_0^t \int_{ \I^2} (j-i) \ind_{\{\alpha_{s^-} =i\}}N(ds,di,dj), \quad \forall t\geq 0. 
\end{equation}
In the following, we will denote $F^{\alpha_{t}}(.) = \sum_{i=1}^{I} F^{i}(.) \ind_{\BRA{\alpha_{t} = i}}$ and $G^{\alpha_{t}}(.) = \sum_{i=1}^{I} G^{i}(.) \ind_{\BRA{\alpha_{t} = i}}$.
\begin{lem}
\label{lemma:eBDSEjump}
Let  ${\PAR{(y^{i}(V_{t}^{v_0}),  z^{i}(V_{t}^{v_0}))_{i \in \I}, \lambda}_{t \geq 0}}$ be the Markovian solution of  the system of ergodic BSDEs \eqref{eq:sebsdeform}, as defined in  Theorem \ref{thm:exsebsde}.
Then, $(y^{\alpha_t}(V_t^{v_0}), z^{\alpha_t}(V_t^{v_0}),  \psi_{t}(\cdot ,\cdot ,V_t^{v_0}) , \lambda)_{t\geq 0 }$ is solution of the following ergodic BSDE with jumps:

\begin{multline}
\label{eq:eBDSEjumps}
Y_{t} = Y_{T} + \int_{t}^{T} \PAR{F^{\alpha_{s}}(V_{s}^{v_0}, Z_{s}) + G^{\alpha_{s}}(Y_{s}) - \sum_{j =1}^I q^{\alpha_{s^-}, j}\psi_s(\alpha_{s^-}, j, V_{t}^{v_0}) }ds \\
- \lambda (T- t) - \int_{t}^{T} (Z_s)^{\top} dW_{s} - \int_{t}^{T} \int_{ \I^2} \psi_s(i, j, V_{t}^{v_0})\ind_{\{\alpha_{s^-} =i\}}\tilde N(ds, di, dj), 
\end{multline}
where 
\begin{equation}
\label{def:jumpsize}
    \psi_{t}(i, j,V_t^{v_0}) := y^j(V_t^{v_0}) - y^i(V_t^{v_0}), 
\end{equation}  and $\tilde N$ the compensated Poisson measure $N(dt,di,dj) -q_{ij} dt\gamma(di)\gamma(dj)$.
 \end{lem}

\subsection{Representation of the ergodic cost}

Finally, the ergodic cost $\lambda$ admits a representation of a weighted space-time average with respect to the invariant measures of the recurrent state Markov chains $V_{t}^{v_0}$. A similar identity was used in the fix -point algorithm of \cite{gobet2024numerical}. It may be used to obtain an approximation of the ergodic cost provided that we know or approximate correctly the invariant distribution $\nu$. 

\begin{prop}
Let 
$\nu$ be the unique invariant probability measure of the diffusion $V^{v_0}$ with infinitesimal generator $\Lr$. Let $\PAR{y^{i}(.), z^{i}(.), \lambda}$ denote the unique Markovian solution to the system of ergodic BSDEs \eqref{eq:sebsdeform} such that $y^{\alpha_{0}}(v_{0}) = y_{0}$. Then the ergodic cost $\lambda$ satisfies, for all $i\in \I$
\begin{eqnarray} \label{eq:lambda_rep}
\lambda =  \int_{\mathbb R^{d'}}  \left(F^{i}(v, z^{i}(v))+\sum_{i \in \I} 
q^{i j}g(y^{j}(v)-y^{i}(v))\right)\nu(dv).
\end{eqnarray}
\end{prop}

\begin{preuve}
The Markovian solution to the system of ergodic BSDEs \eqref{eq:sebsdeform} verifies the following equation
\begin{eqnarray} \label{eq:ebsdemark}
y^{i}(v_0) &=& \mathbb E\left[y^{i} (V_T^{v_0})+ \int_{0}^{T} \PAR{F^{i}(V_{s}^{v_0}, z^{i}(V_s^{v_0})) + \sum_{j=1}^{I} q^{ij} g(y^{j}(V_s^{v_0}) - y^{i}(V_s^{v_0}))-\lambda} ds \right].
\end{eqnarray}
Integrating with respect to the measure $\nu$ yields
\begin{eqnarray*} 
\int_{\mathbb R^{d'}}y^{i}(v_0)\nu(dv_0) &=& \int_{\mathbb R^{d'}}\mathbb E\left[y^{i} (V_T^{v_0})\right]\nu(dv_0)\\
&+& \int_{\mathbb R^{d'}}\mathbb E\left[\int_{0}^{T} \PAR{F^{i}(V_{s}^{v_0}, z^{i}(V_s^{v_0})) + \sum_{j=1}^{I} q^{ij} g(y^{j}(V_s^{v_0}) - y^{i}(V_s^{v_0}))-\lambda} ds \right]\nu(dv_0).
\end{eqnarray*}
Then, we can conclude \eqref{eq:lambda_rep} by applying Fubini theorem and the definition of the invariant probability measure $\nu$.

\end{preuve}

\section{Deep learning algorithms for the simulation of systems of eBSDEs} \label{section:algoswitch}

In this section, we focus on the numerical simulation of the Markovian solution to the system of ergodic BSDEs introduced above. We present two deep learning-based approaches. The first is a deep BSDE scheme of \emph{locally additive} type, in which the solution and its gradient are represented by neural networks trained against a sum of local losses, each accumulating the dynamics up to the terminal condition, following the methodology of LaDBSDE in \cite{kapllani2020deep}. The second approach relies on a Galerkin-type approximation, see \cite{sirignano2018dgm}. Deep learning methods for BSDEs have attracted considerable attention in recent years, mainly because of their ability to handle high-dimensional nonlinear PDEs through probabilistic BSDE representations; see, for instance, \cite{han2017deep}, \cite{chan2019machine}, \cite{hure2020deep}, \cite{germain2021neural}, and \cite{kapllani2020deep}.

In the present framework, we adapt these ideas to systems of ergodic BSDEs, in the spirit of the approach developed in \cite{deepbroux2024}. We first establish a connection between the regime-switching system of eBSDEs and a system of eBSDEs with random terminal time. This representation will serve as the probabilistic foundation for the deep BSDE scheme introduced below. The Galerkin-type method follows a different numerical strategy and is presented independently afterwards.

\subsection{Connection with a system of eBSDEs with random terminal time} \label{sect:connectionrdtbsde}

In the spirit of \cite{deepbroux2024}, we exploit the recurrence property of the stochastic factor $V^{v_0}$ to establish a correspondence between the system of ergodic BSDEs \eqref{eq:sebsdeform} and a multidimensional BSDE with random terminal time. Unless stated otherwise, we assume throughout this section that the stochastic factor is one-dimensional, namely $d'=1$.  We fix a minimal deterministic horizon $T_0$ and define the random horizon $\tau$ as the first return time after $T_0$ of the diffusion $V$ to its initial value $v_0$:
\begin{eqnarray} \label{eq:deftauswitch}
\tau = \inf \BRA{t > T_{0}, , V_{t}^{v_0} = v_{0}}.
\end{eqnarray}

Let $i_0 \in \I$ and $v_0, y_0 \in \mathbb R$. Under Assumptions \ref{ass:weakdissass} and \ref{ass:FGgrowthass}, there exists a unique solution
$
\left((y^{i}(V_t^{v_0}), z^{i}(V_t^{v_0}))_{i \in \I}, \lambda \right)_{t\geq 0}
$
to the system of ergodic BSDEs \eqref{eq:sebsdeform} such that $y^{i_0}(v_0)=y_0$. Moreover, for all $i\in \I$, the function $y^{i}$ has sublinear growth with respect to $v$, the function $z^{i}$ is bounded by
$
Z_{\max}:=\NRM{\kappa} \frac{C_{v}}{C_{\mu} - C_{v}},
$
and the differences $y^{i}-y^{j}$ are uniformly bounded.\\

By construction,
$
\left((y^{i}(V_t^{v_0}), z^{i}(V_t^{v_0}))_{i \in \I}, \lambda \right)_{t\geq 0}
$
also solves the following ``ergodic'' BSDE with random terminal time and the {\it{same initial and terminal condition}}: for each component $i=1,\ldots,I$,
\begin{align} \label{eq:rdtBSDEswitch}
Y^{r, i}_{t}
&= Y^{r, i}_{\tau}+\int_{t}^{\tau} \PAR{F^{i}(V_{s}^{v_0}, Z^{r, i}_{s}) + G^{i}(Y_{s}^{r})}ds - \lambda (\tau- t) - \int_{t}^{\tau} (Z^{r ,i}_{s})^{\top} dW_{s}, \\
  Y_{\tau}^{r, i} &= Y^{r,i}_0,
  \qquad
  Y^{r,i_0}_\tau = y_0. \nonumber
  \end{align}
  Here, $\tau$ is the return time defined in \eqref{eq:deftauswitch}. The term ``ergodic'' BSDE with random terminal time is a slight abuse of terminology, since the equation is no longer formulated on an infinite time horizon. We use it to emphasize that the ergodic cost $\lambda$ remains part of the unknowns in \eqref{eq:rdtBSDEswitch}. The infinite horizon in \eqref{eq:sebsdeform} is replaced by the random terminal time $\tau$, together with the additional constraint that the initial and terminal values coincide, namely
  $
  Y_{0}^{r,i}=Y_{\tau}^{r,i}, \,\ i \in \I.
  $

Conversely, solutions
$
\left((Y^{r,i},Z^{r,i})_{i\in\I},\lambda\right)_{t\in[0,\tau]}
$
of \eqref{eq:rdtBSDEswitch} can be studied in their own right. The main difficulty, compared with the one-dimensional case, is that the normalization cannot be imposed componentwise. Indeed, only the value of the solution at one point $v_0$ \textit{and} in one state $i_0$, say $Y^{r,i_0}_0=Y^{r,i_0}_\tau=y_0$, can be fixed. On the other hand, the  remaining initial values $Y^{r,i}_0$, for $i\neq i_0$, cannot be chosen due the coupling term $G=(G^i)_{i\in\I}$ and the shared parameter $\lambda$. Hence, the random-time formulation does not reduce to a collection of independent scalar BSDEs, the coupling must be handled simultaneously with the ergodic constant $\lambda$ and the periodic-type constraint $Y^{r,i}_0=Y^{r,i}_\tau$ for all $i\in\I$.

The following theorem states the existence and uniqueness of a Markovian solution to \eqref{eq:rdtBSDEswitch}. It ensures that this solution indeed coincides on $\SBRA{0,\tau}$ with the unique solution of the system of ergodic BSDEs \eqref{eq:sebsdeform}.

Recall that $C_Y$ is defined in \eqref{eq:unifboundy}. Let $K_{z}:= C_{z} (1 + 2Z_{\max})$ be the Lipschitz constant with respect to $z$ of the truncated driver $F \circ \varphi_{Z_{\max}}$, and let $K_G$ be the Lipschitz constant of $\underset{i \in \I}{\max}G^i$ on the set $\left\{ x \in \mathbb{R}^I \, ; \, |x^i - x^j| \leq C_Y, \; \forall \, i,j \in \mathcal{I} \right\}$.

\begin{thm}
\label{thm:BSDERandomtimeswitch}
Suppose that Assumption \ref{ass:weakdissass} and \ref{ass:FGgrowthass} and \ref{ass:MC} hold. Assume that the random horizon $\tau$ admits an exponential moment of order $\Gamma > 2K_{G} + K_{z}^{2}$. \\ 
Then, the ``ergodic'' BSDE with random horizon  \eqref{eq:rdtBSDEswitch} with $Y_{\tau}^{r, i} = Y^{r,i}_0$ for all $i \in \I$, and $Y^{r,i_0}_\tau  = y_0$, 
admits a unique Markovian solution $((y^{i}(.), z^{i}(.))_{i \in \I}, \lambda)$, such that for any  $\gamma < \Gamma$, \, $ y(V_{t}^{v_{0}}) \in \Sr^{2}(\gamma, \tau)$, $z(V_{t}^{v_{0}}) \in \M(\gamma, \tau)$. 

\noindent In addition, if for all $i, \, j \in \I$, $y^{i}$ is sublinear, differences $y^{i} - y^{j}$ are uniformly bounded by $C_{Y}$, and $z^{i}$ is bounded by $Z_{\max}$, then the solution $((y^{i}(.), z^{i}(.))_{i \in \I}, \lambda)$ coincides on $\SBRA{0, \tau}$ with the unique Markovian solution of the  system of eBSDEs \eqref{eq:sebsdeform}, as defined in Theorem \ref{thm:exsebsde}, verifying $y^{i_0}(v_0) = y_0$. 
\end{thm}

\begin{rmq}
    Sufficient conditions for the exponential integrability of the hitting time $\tau$ can be obtained as a straightforward extension of Lemma 3.1 in \cite{deepbroux2024}.
\end{rmq}

\begin{preuve}
\textbf{Existence -} Let $\gamma < \Gamma$. By construction,  the unique solution $\PAR{(y^{i}(V_{t}^{v_{0}}), z^{i}(V_{t}^{v_{0}})))_{i \in \I}, \lambda}_{t \geq 0}$ of the system of ergodic BSDEs \eqref{eq:sebsdeform} is also solution of the BSDE with random terminal time \eqref{eq:rdtBSDEswitch}, with  $z =(z^i)_{i \in \I}$ bounded by $Z_{\max}$ so that $Z \in \M(\gamma, \tau)$.  For all $i \in \I$, $y^{i}$ has sublinear growth. Hence, there exists a constant $C > 0$ such that for all $i \in \I$ 
\begin{eqnarray*}
\e \SBRA{e^{\gamma \tau} \underset{0 \leq t \leq \tau}{\sup} \ABS{Y_{t}^{i}}^{2}} \leq 2C^{2} \e \SBRA{ e^{\gamma \tau}} + 2C^{2}  \e \SBRA{ e^{\gamma \tau} \underset{0 \leq t \leq \tau}{\sup} \ABS{V_{t}^{v_{0}}}^{2}}.
\end{eqnarray*}
Under the assumption that $\tau$ admits exponential moments of order $\Gamma > \gamma$, the first expectation on the right side is finite. Let $p > 1$ and $q = \frac{p}{p-1}$ be conjugate exponents, such that $p \gamma < \Gamma$. By Hôlder's inequality,
\begin{eqnarray*}
\e \SBRA{e^{\gamma \tau} \underset{0 \leq t \leq \tau}{\sup} \ABS{V_{t}^{v_{0}}}^{2}} \leq \e \PAR{\SBRA{e^{p \gamma \tau}}}^{\frac{1}{p}} \PAR{\e \SBRA{\underset{0 \leq t \leq \tau}{\sup} \ABS{V_{t}^{v_{0}}}^{2q}}}^{\frac{1}{q}}.
\end{eqnarray*}
The first factor is finite by assumption, while the second one is finite by Lemma \ref{lem:runningsupV}. In turn $Y~\in~\Sr^{2}~(\gamma, \tau)$.

\noindent
\textbf{Uniqueness of $\lambda$ -} Let $\PAR{(Y^{r,i}, Z^{r,i})_{i \in \I}, \lambda}_{t\in [0,\tau]}$ and $\PAR{(\overline{Y}^{r,i}, \overline{Z}^{r,i})_{i \in \I}, \overline{\lambda}}_{t\in [0,\tau]}$ be two Markovian solutions of the ergodic system of BSDEs with random terminal time \eqref{eq:rdtBSDEswitch}. Let $\alpha$ be the CTMC defined in \eqref{eq:defalpha}. This will serve only as a representation tool for uniqueness and without loss of generality, we can assume it starts from any state $l \in \I$. Define for all $0\leq t \leq \tau$, 
\begin{align*}
  &  (Y_t^r, Z^r_t):= (Y^{r,\alpha_t}, Z^{r,\alpha_t}), \quad (\overline{Y}^r_t, \overline{Z}_t^r):= (\overline{Y}^{r,\alpha_t}, \overline{Z}^{r,\alpha_t}), \text{ and }\\
&  \psi_{t}^r(i, j,V_t^{v_0}) := Y^{r,j}_{t} - Y^{r,i}_{t},   \quad \overline{\psi}^r_{t}(i, j,V_t^{v_0}) := \overline{Y}^{r,j}_{t}- \overline{Y}^{r,i}_{t}.
\end{align*}
Then, using the same notation as in Lemma \ref{lemma:eBDSEjump}, $(Y^r,Z^r, \psi^r, \lambda)$ and $(\overline{Y}^r,\overline{Z}^r,\overline{\psi}^r, \overline{\lambda})$ are solutions on $[0,\tau]$ of 
\begin{align*}
  & dY_{t} = -   \PAR{F^{\alpha_{t}}(V_{t}^{v_0}, Z_{t})  +   G^{\alpha_{t}}(Y_{t})}  dt 
  + \lambda dt  + (Z_t)^{\top} dW_{t} +  \int_{ \I^2} \psi_t(i, j,V_t^{v_0})\ind_{\{\alpha_{t^-} =i\}} N(dt, di, dj), \quad \forall 0 \leq t \leq \tau.
\end{align*}
Let $\Delta Y_{t}^r = Y_{t}^{r} - \overline{Y}_{t}^{r}, \, \Delta Z_{t}^r = Z_{t}^{r} - \overline{Z}_{t}^{r}$, $\Delta \lambda = \lambda - \bar \lambda$, and $\Delta \psi_t^r(i,j) =  \psi_{t}^r(i, j) -  \overline{\psi}^r_{t}(i, j)$.  \\
The equation verified by $\Delta Y^r$  can be linearized by using a change of measure as  in the proof of Theorem 3.2 in \cite{hu2020systems}.
Indeed, using the previous equation and by definition of $(G^i)_{i\in \I}$ in Assumption \ref{ass:FGgrowthass} (iii), we have $ \forall \;  0 \leq t \leq \tau$
\begin{align}
\label{eq:deltaYr}
   d\Delta Y_{t}^r &  = -   (F^{\alpha_{t}}(V_{t}^{v_{0}}, Z^r_{t}) -  F^{\alpha_{t}}(V_{t}^{v_{0}},\overline{Z}_{t}^r)) dt + \Delta \lambda dt \\
 \nonumber & \qquad  \qquad  -   \sum_{j =1}^I q^{\alpha_{t^-}, j} \left(g(\psi_{t}^r(\alpha_{t^-}, j)) - g(\overline{\psi}^r_{t}(\alpha_{t^-}, j))  - \Delta \psi_t^r(\alpha_{t^-}, j)\right)   dt 
   \\
  & \qquad  \qquad   + (\Delta Z_t^r )^{\top} dW_{t} +  \int_{ \I^2} \Delta \psi_s^r(i, j)\ind_{\{\alpha_{t^-} =i\}}\tilde N(dt, di, dj)\\
\nonumber & = \Delta \lambda dt +  (\Delta Z_t^r)^{\top} (dW_{t} - \gamma_{t}^{F} dt) +  \int_{ \I^2} \Delta \psi_t^r(i, j) \ind_{\{\alpha_{t^-} =i\}} \left(\tilde N(dt, di, dj) -\gamma_{t}^{g}(i,j) dt \gamma(di)\gamma(dj) \right),
\end{align}
with
\begin{align*}
& \gamma_{t}^{F} = \frac{F^{\alpha_{t}}(V_{t}^{v_{0}}, Z_{t}^{r, \alpha_{t^-}}) - F^{\alpha_{t}}(V_{t}^{v_{0}}, \overline{Z}_{t}^{r, \alpha_{t^-}})}{\NRM{\Delta Z_{t}^r}^{2}} \Delta Z_{t}^r \ind_{\BRA{ \Delta Z_{t}^r \neq 0}}, \quad \text{ and}    \\
&\gamma_{t}^{g}(i,j) = q^{i, j} \frac{g(\psi_{t}^r(i, j)) - g(\overline{\psi}^r_{t}(i, j))  - \Delta \psi_t^r(i, j)}{ \Delta \psi_t^r(i, j)} \ind_{\BRA{  \Delta \psi_t^r(i, j) \neq 0}}. 
\end{align*}
By Assumption \ref{ass:FGgrowthass}  and the boundedness of $Z^r$,  $\overline{Z}^r$, $(Y^{r,i} - Y^{r,j})$ and $(\overline{Y}^{r,i} - \overline{Y}^{r,j})$, the processes $\gamma_{t}^{F}$ and $(\gamma_{t}^{g}(i,j))_{i,j \in \I^2}$ are  uniformly bounded. Hence, by Girsanov's theorem, for each initial state $\alpha_0 = l \in \I$ and all $T\geq 0$, there exists an equivalent probability measure  $\tilde \p_{l}$  under which  $(\Delta Y_{t\wedge \tau}^r - \Delta \lambda (t \wedge \tau))_{0\leq t\leq T}$ is a martingale (for ease of notations, we omit the $v_0$ in $\tilde \p_{l}$). Consequently 

\begin{eqnarray*}
   \Delta Y^{r, l}_{0} &=& \tilde \e_{l}[\Delta Y^r_{T \wedge \tau}] -  \Delta \lambda \tilde \e_{l}[T \wedge \tau]. \\
   &=& \tilde \e_{l}[\Delta Y^r_{\tau} \ind_{T \geq \tau}] + \tilde \e_{l}[\Delta Y^r_{T} \ind_{\tau > T}] -  \Delta \lambda \tilde \e_{l}[T \wedge \tau]. 
\end{eqnarray*}
Since  each solution satisfies $Y_{\tau}^{r, i} = Y_{0}^{r, i}$, for all $i \in \I$, the differences also satisfy $\Delta Y_{\tau}^{r, i} = \Delta Y_{0}^{r, i}$. Therefore
\begin{eqnarray}
   \Delta Y^{r, l}_{0} = \sum_{k \in \I} \tilde{\p}_{l}(\tau < T , \alpha_{ \tau} = k)  \Delta Y^{r, k}_{0}  + \tilde \e_{l}[\Delta Y^r_{T} \ind_{\tau > T}] -  \Delta \lambda \tilde \e^{l}[T \wedge \tau]. \label{eq:decompesp_uniqlambd}
\end{eqnarray}
Using the sublinear growth property of $\Delta Y$, there exists a constant $C$ such that
\begin{eqnarray}
\tilde \e_{l}[\Delta Y^r_{T} \ind_{\tau > T}] &\leq& \tilde \e_{l}[\ABS{\Delta Y^r_{T}}^{2}]^{\frac{1}{2}} \tilde{\p}_{l}(\tau > T)^{\frac{1}{2}} \nonumber \\ 
    &\leq& C (1 + \tilde \e_{l}[\ABS{V_{T}^{v_0}}^{2}])^{\frac{1}{2}} \tilde{\p}_{l}(\tau > T)^{\frac{1}{2}}. \label{eq:uniqtausupT}
\end{eqnarray}
Under  $\tilde \p_{l}$, the process $(V^{v_0},\alpha)$ satisfies
\begin{align*}
    dV_t^{v_0} & =  \left( \mu(V_t^{v_0}) + \gamma_t^F\right)dt + \kappa^\top (d{W}_t - \gamma_t^Fdt) \\
    d\alpha_t & =  \int_{\I^2} (j-i) \ind_{\{\alpha^- = i\}} N(ds, di, dj),  
\end{align*}
with ${W} - \int_0^\cdot \gamma_s^F ds$ is a Brownian motion under $\tilde \p_{l}$ and $N$ is a point process with intensity $(q^{ij} + \gamma_s^g(i,j))_{s\geq 0, (i,j) \in \I^2}$ under $\tilde{\p}_l$. \\
Moreover, $\gamma_s^F$ can be written as a uniformly bounded function of $V_s^{v_0}$. By Theorem 2.6 and Lemma 3.4 in \cite{debussche2011ergodic}, this yields that $V^{v_0}$ is recurrent under $\tilde{\p}_l$. Hence
\begin{equation*}
    \tilde{\p}_{l}(\tau > T) \xrightarrow[T \to\infty]{}0. 
\end{equation*}
Furthertmore, by Proposition 5 in \cite{hu2019ergodic}, $$\underset{T \geq 0}{\sup}\, \tilde{\e}_l \SBRA{\ABS{(V_{T}^{v_0})^{2}}} < \infty.$$ 
Thus, taking the limit as $T$ goes to infinity in \eqref{eq:uniqtausupT} leads  
\begin{eqnarray*}
    \tilde \e_{l}[\Delta Y^r_{T} \ind_{\tau > T}] \xrightarrow[T \to\infty]{}0. 
\end{eqnarray*}
By the equivalence of $\p_l$ and $\tilde{\p}_l$ on $\mathcal{F}_T$, 
\begin{equation*}
    \tilde{\p}_{l}(\tau < T, \alpha_{ \tau} = k) >0  \Longleftrightarrow \p_{l}(\tau < T, \alpha_{ \tau} = k) >0
\end{equation*}
By recurrence of $(V^{v_0}, \alpha)$ on $\p_l$, $\p_{l}(\tau < T, \alpha_{T \wedge \tau} = k) >0$, for $T$ large enough. By monotone convergence, this yields that 
$$\tilde{\p}_{l}(\tau < T, \alpha_{ \tau} = k) \uparrow \tilde{\p}_{l}(\tau < \infty, \alpha_{ \tau} = k)>0 .$$ 
Let $\tilde P = (\tilde P_{lk})_{1 \leq l, k \leq I} := (\tilde{\p}_{l}(\tau < \infty, \alpha_{ \tau} = k))_{k,l\in \I}$. Then, 
$\tilde P$ is a strictly positive stochastic matrix on the finite state space $\I$. It therefore admits a unique invariant measure $\tilde{\pi} = (\tilde{\pi}^l)_{l\in \I}$. Combining this with \eqref{eq:decompesp_uniqlambd}  and letting $T \to \infty$, we obtain  
\begin{align*}
\sum_{l \in \I} \tilde{\pi}_{l} \Delta Y_{0}^{r, l} %
& = \sum_{k \in \I} \left( \sum_{l\in \I} \tilde{\pi}_{l}\tilde P_{lk}\right)  \Delta Y^{r, k}_{0} -  \Delta \lambda \sum_{l \in \I} \tilde{\pi}_{l} \tilde \e_{l}[\tau].
\end{align*}
Since $\tilde{\pi} \tilde P = \tilde{\pi}$, this becomes
\begin{eqnarray*}
\sum_{l \in \I} \tilde{\pi}_{l} \Delta Y_{0}^{r, l} = \sum_{k \in \I} \tilde{\pi}_{k} \Delta Y^{r, k}_{0} -  \Delta \lambda  \sum_{l \in \I} \tilde{\pi}_{l} \tilde \e^{l}[\tau].
\end{eqnarray*}
Hence, $\Delta \lambda  \underset{l \in \I}{\sum} \tilde{\pi}_{l} \tilde \e_{l}[\tau] = 0$. Finally, since $\tilde \e_{l}[\tau] > T_0$ for every $l \in \I$, we conclude that  $\Delta \lambda = 0$.

\noindent
\textbf{Uniqueness of terminal values -} Since $\Delta  \lambda = 0$, identity \eqref{eq:decompesp_uniqlambd} rewrites in vector form
\begin{eqnarray*}
    \Delta Y_{0}^r = \tilde{P} \Delta Y_{0}^r,
\end{eqnarray*}
where $\Delta Y_{0}^{r} = \PAR{\Delta Y_{0}^{r, l}}_{1 \leq l \leq I}$. Thus $\Delta Y_{0}^{r}$ is a $\tilde{P}$-harmonic function on the finite state space $\I$. Since $P$ is irreducible, Lemma 1.16 in \cite{levin2017markov} implies that $\Delta Y_{0}^{r}$ is constant on $\I$. By \eqref{eq:rdtBSDEswitch}, $Y^{r,i_0} = \bar{Y}^{r,i_0} = y_0$. Hence, for every $l \in \I$,
\begin{eqnarray}
    \Delta Y_{0}^{r, l} = \Delta Y_{0}^{r, i_{0}} = 0.
\end{eqnarray}
Finally, since $\Delta Y_{\tau}^{r, l}=  \Delta Y_{0}^{r, l} $ by \eqref{eq:rdtBSDEswitch}, we also get that  for all $l \in \I$,  $ \Delta Y_{\tau}^{r, l}  = 0.$

\noindent
\textbf{Uniqueness of Markovian solution $(y^{i}(.), z^{i}(.))_{i \in \I}$ - } Define for all $i \in\I$, $y \in \R^{I}$, a truncated version of the coupling term in the generator, namely 
\begin{eqnarray*}
G^{i}_{C_{Y}}(y) = \sum_{j \neq i} q^{ij} g \circ \varphi_{C_{Y}}(y^{j} - y^{i}).
\end{eqnarray*}
The rest of the proof follows by applying Theorem $3.1$ in \cite{pardoux1998backward} to the BSDE with random terminal time $\tau$, terminal condition $Y_{\tau}^{r}$ (unknown but unique from the last step) and generator $F \circ \varphi_{Z_{\max}} (v, z) + G_{C_{Y}}(y) - \lambda$, where $\lambda$ is fixed. Under Assumption \ref{ass:FGgrowthass}, the generator $F \circ \varphi_{Z_{\max}} + G_{C_{Y}}$ is Lipschitz in $z$ and satisfy the usual monotonicity condition in $y$, with monotonicity constant $K_{G}$. By assumption, the random  horizon $\tau$ admits exponential moments of order $\Gamma > 2K_{G} + K_{z}^{2}$. Theorem $3.1$ in \cite{pardoux1998backward} applies and ensures existence and uniqueness of a Markovian solution  $((Y^{r, i}, Z^{r, i})_{i \in \I}) \in \Sr^{2}(\gamma, \tau) \times \mathcal{M}(\gamma, \tau)$. From the first step of the proof, this solution coincides necessarily with the unique Markovian solution $\PAR{(y^{i}(V_{t}), z^{i}(V_{t}))_{i = 1, ..., I}, \lambda}_{t \geq 0}$ of the system of ergodic BSDEs \eqref{eq:sebsdeform}, which concludes.
\end{preuve}

The solution of the system of ergodic BSDE \eqref{eq:sebsdeform} thus coincides with the Markovian solution $\PAR{(Y^{r, i}, Z^{r, i})_{i \in \I}, \lambda}$ of the ergodic BSDE with random time horizon and fixed initial condition \eqref{eq:rdtBSDEswitch} on $\SBRA{0, \tau}$. We will omit the subscript $r$ in the sequel. This point of view provides our simulation problem with a random horizon $\tau$ allowing to adapt numerical schemes introduced in \cite{deepbroux2024}.

\subsection{Locally additive deep solver}

We can now introduce the  locally additive deep solver for approximating the solution $\PAR{(Y_t^{i} , Z_t^{i})_{i \in \I}, \lambda}_{t\geq 0 }$ of the random time horizon ergodic BSDE \eqref{eq:rdtBSDEswitch}, associated with the forward stochastic factor given by  \eqref{Vequation}. 
The main idea is to minimize a global loss obtained from the aggregation of local residuals of a forward discretization of the BSDE. In contrast with the backward formulation from \cite{kapllani2020deep}, the forward discretization is more convenient for simulating systems of ergodic BDSEs.  In our multidimensional setting, the terminal conditions of the random time horizon ergodic BSDE \eqref{eq:rdtBSDEswitch} are unknown for components $i\neq i_0$. However, we address  this issue by incorporating the constraint $Y^{i}_\tau = Y_0^{i}$ in the  loss function, while the normalization constraint $Y^{i_0} = y_0$ is included as a penalization term. 

The component $Y$ is represented by a neural network, common to all time steps, while $Z$ can either be computed with automatic differentiation or with another neural network. In the context of systems of ergodic BSDE, the ergodic cost $\lambda$  will be estimated as a trainable parameter of the model, denoted by $\bar{\lambda}$. 

\paragraph{Stochastic factor and random horizon approximations -} The forward stochastic factor $V^{v_0}$ is approximated by a Euler discretization on a time grid $\pi = (t_i)_{i\geq 0}$ with $t_0=0$ and time step $h$. Denoting for all $i \geq 0$, $\Delta W_{t_{i}} = W_{t_{i+1}} - W_{t_{i}}$ the Brownian increment at time $t_{i}$ :
\begin{eqnarray*}
\overline{V}_{t_{i+1}} &=& \overline{V}_{t_{i}}+ \mu(\overline{V}_{t_{i}}) h + \kappa \Delta W_{t_i}, \\
\overline{V}_{0} &=& v_{0}. \nonumber
\end{eqnarray*}
We denote by  $\overline{\tau}$ the first hitting time in the time grid  of $\overline{V}$ to $v_{0}$ after $T_0$. More precisely, assuming $T_0 \in \pi$, we set
\begin{eqnarray}
    \overline{\tau} = \inf \BRA{ t_{i} > T_0, \, t_{i} \in \pi \, ; (\overline{V}_{T_0} - v_{0})(\overline{V}_{t_{i}} - v_{0}) \leq 0}.
\end{eqnarray}
We denote $N_{\overline{\tau}} = \frac{\overline{\tau}}{h}$ the corresponding number of time steps and  set $V_{\overline{\tau}}= v_0$.

\paragraph{Forward discretization of the eBSDE.}  Let $\Y^{\theta_{1}} : \R \to \R^{I}$ and $\Z^{\theta_{2}} : \R \to \R^{d \times I}$ be two neural networks, shared across all time steps, intended to approximate the Markovian components $y$ and $z$ of the system of ergodic BSDEs, associated with \eqref{eq:rdtBSDEswitch}. The initial value of the forward discretization is taken as $\Y^{\theta_{1}}(v_{0})$. Without loss of generality, we can assume that $i_0=1$.

Given a realization of the Euler scheme $\overline{V}$, we consider a forward discretization of Equation \eqref{eq:rdtBSDEswitch} on the time grid $\pi$, 
\begin{eqnarray} \label{eq:discrebsdeswitch}
\overline{Y}_{t_{i+1}}^{\theta_{1}, \bar{\lambda}} = \overline{Y}_{t_{i}}^{\theta_{1}, \bar{\lambda}} - F(\overline{V}_{t_{i}}, \Z^{\theta_{2}}(\overline{V}_{t_{i}})) h - G(\overline{Y}_{t_{i}}^{\theta_{1},  \bar{\lambda}})h + \bar{\lambda} h + \Z^{\theta_{2}}(\overline{V}_{t_{i}})^{\top}\Delta W_{t_{i}}, \quad \forall \; 0\leq i \leq N_{\overline{\tau}}.
\end{eqnarray}
The construction of local residuals relies on the iteration of the time discretization \eqref{eq:discrebsdeswitch}. In fact, for $1\leq i \leq N_{\overline{\tau}}$,
\begin{eqnarray} \label{eq:iterative_timediscr_switch}
\overline{Y}_{t_{i}}^{\theta_1, \bar{\lambda}} = \Y^{\theta_{1}}(v_{0}) + \bar{\lambda}t_i - \sum_{k=0}^{i-1} \SBRA{ \PAR{F(\overline{V}_{t_{k}}, \Z^{\theta_2}(\overline{V}_{t_{k}})) + G(\overline{Y}_{t_{k}}^{\theta_{1}, \bar{\lambda}})} h  - \Z^{\theta_2}(\overline{V}_{t_{k}})^{\top} \Delta W_{t_{k}}}.
\end{eqnarray}
\paragraph{Loss function.} The error at each time step is evaluated by computing the squared difference between the network prediction $\Y^{\theta_{1}}(\overline{V}_{t_{i}})$ with the value \eqref{eq:iterative_timediscr_switch} obtained by iteration of the forward discretization with the initial condition. For $1\leq i \leq N_{\overline{\tau}}$, define
\begin{eqnarray} \label{eq:local_err_LA}
\phi_{t_{i}}(\overline{V}, \theta_{1}, \theta_{2}, \lambda) =  \sum_{k=0}^{i-1} \SBRA{\PAR{F(\overline{V}_{t_{k}}, \Z^{\theta_{2}}(\overline{V}_{t_{k}})) + G(\Y^{\theta_{1}}(\overline{V}_{t_{k}}))}h  - \Z^{\theta_{2}}(\overline{V}_{t_{k}})^{\top} \Delta W_{t_{k}}} - \bar{\lambda} t_{i}.
\end{eqnarray}
The quantity $\phi_{t_{i}}$ represents the cumulative increment between time $0$ and time $t_{i}$ induced by the forward discretizations. The residual at time $t_{i}$ is then defined by
\begin{eqnarray} \label{eq:generr_term}
\Err_{t_{i}} = \NRM{\Y^{\theta_{1}}(\overline{V}_{t_{i}}) + \phi_{t_{i}}(\overline{V}, \theta_{1}, \theta_{2}, \lambda) - \Y^{\theta_{1}}(v_{0})}^{2},
\end{eqnarray}
At the initial time, the normalization condition is enforced through
\begin{eqnarray}
\Err_{0} = \ABS{\Y^{1, \theta_{1}}(v_{0}) - y_{0}}^{2}.
\end{eqnarray}
At the terminal time $t_{N_{\overline{\tau}}} = \overline{\tau}$, the condition $V_{\overline{\tau}} = v_{0}$ implies that the initial and terminal value cancels out in \eqref{eq:generr_term}, leading to the residual
\begin{eqnarray*}
\Err_{N_{\bar{\tau}}} = \NRM{\phi_{t_{N_{\tau}}}(\overline{V}, \theta_{1}, \theta_{2}, \bar{\lambda})}^{2}.
\end{eqnarray*}
These errors are then aggregated into a global loss function, which takes the form
\begin{eqnarray} \label{eq:loc_finallossswitch}
L_{\loc}(\theta_{1}, \theta_{2}, \bar{\lambda}) = \e \left[ \ABS{\Y^{1, \theta_{1}}(v_{0}) - y_{0}}^{2} + \displaystyle\sum_{i=1}^{N_{\tau}} \NRM{\Y^{\theta_{1}}(\overline{V}_{t_{i}}) + \phi_{t_{i}}(\overline{V}, \theta_{1}, \theta_{2}, \bar{\lambda})  - \Y^{\theta_{1}}(v_{0})}^{2} \right]
\end{eqnarray}
In practice, this expectation is approximated by Monte Carlo over a set of trajectories $\PAR{\bar{V}^{j}}_{j = 1, \dots B}$, up to their  return time $\PAR{\overline{\tau}_{j}}_{j = 1, \dots B}$, yielding the empirical loss
\begin{eqnarray} \label{eq:emploc_finallossswitch}
L_{\loc}^{B}(\theta_{1}, \theta_{2}, \bar{\lambda}) = \frac{1}{B} \sum_{j=1}^{B} \left[ \ABS{\Y^{1, \theta_{1}}(v_{0}) - y_{0}}^{2} + \displaystyle\sum_{i=1}^{N_{\bar{\tau}_{j}}} \NRM{\Y^{\theta_{1}}(\overline{V}_{t_{i}}^j) + \phi_{t_{i}}(\overline{V}^{j}, \theta_{1}, \theta_{2}, \bar{\lambda})  - \Y^{\theta_{1}}(v_{0})  }^{2} \right]
\end{eqnarray}
\paragraph{Automatic differentiation variant.} A variant consists of computing $Z$ using automatic differentiation, rather than through an additional neural network. In that case, one sets
\begin{eqnarray}
\Z_{t_{i}}^{\theta_{1}} = \left. \kappa \frac{ \partial \Y^{\theta_{1}}(v)}{\partial v} \right|_{v = \overline{V}_{t_{i}}} \in \R^{d\times I },
\end{eqnarray}
and optimizes only over $(\theta_{1}, \overline{\lambda})$. We refer to this modified algorithm as the ADLAeBSDE solver. 

\paragraph{Implementation.} The complete procedure is summarized in Algorithm \ref{alg:locebsdeswitch}. At each gradient step, one first simulates a batch of size $B$ of Euler trajectories then evaluates the empirical loss $L_{\loc}^{B}$, and updates the parameters by stochastic gradient descent.

\begin{algorithm}[!ht]
\DontPrintSemicolon
Let $\Y^{\theta_{1}}$ be a neural network defined on $\R$, valued in $\R^I$ with parameters $\theta_{1}$ and $\Z^{\theta_{2}}$ be a neural network defined on $ \R$, valued in $\R^{d \times I}$, with parameters $\theta_{2}.$ Let $\theta^{0} = (\theta_{1}^{0}, \theta_{2}^{0}, \bar{\lambda}^{0}) \in \R^{3}$ be the initialization of the neural network and ergodic cost parameters. Define $N_{T_{0}} = \lfloor \frac{T_{0}}{h} \rfloor$. \\
\For{$m=0$, ..., $M$}{
\For{$j=1, ..., B$}{
\For{$k\in \BRA{0, ..., N_{T_{0}}}$, starting from $\overline{V}_{0}^{j} = v_{0}$}{
Sample $\Delta W_{t_{k}}^{j} \sim \N(0, h I_{d})$ and compute \\
$\overline{V}_{t_{k+1}}^{j} = \overline{V}_{t_{k}}^{j} + \mu( \overline{V}_{t_{k}}^{j}) h + \kappa^{\top} \Delta W_{t_{k}}^{j},$ \\
}
Let $N_{j} = N_{T_{0}}+1$. \\
\While{$(\overline
{V}_{t_{N_{T_{0}}}}^{j} - v_{0})(\overline{V}_{t_{N_{j}}}^{j} - v_{0}) > 0$}{
Sample $\Delta W_{t_{N_{j}}}^{j} \sim \N(0,h I_{d})$ and compute \\
$\overline{V}_{t_{N_{j}+1}}^{j} = \overline{V}_{t_{N_{j}}}^{j} + \mu( \overline{V}_{t_{N_{j}}}^{j}) h + \kappa^{\top} \Delta W_{t_{N_{j}}}^{j},$ \\
$N_{j} = N_{j}+1$
}
Set, $h N_{j} = \overline{\tau_{j}}$.
}
\For{$j=1, ..., B$}{
Set, $\phi_{t_{0}} = 0$. \\
\For{$k\in \BRA{1, ..., N_{j}}$, starting from $\overline{Y}_{0}^{j} = \Y^{\theta_{1}^{m}}(v_{0})$}{
$ \psi_{t_{k}}^{\theta^{m}, j} =  h F(\overline{V}_{t_{k-1}}^{j}, \Z^{\theta_{2}^{m}}(\overline{V}_{t_{k-1}}^{j})) + h G (\Y^{\theta_{1}^{m}}(\overline{V}_{t_{k-1}}^{j})) - \bar{\lambda}^{m} h - \Z^{\theta_{2}^{m}}(\overline{V}_{t_{k-1}}^{j})^{\top} \Delta W_{t_{k-1}},$ \\
$\phi_{t_{k}}^{\theta^{m}, j} = \phi_{t_{k-1}}^{\theta^{m}, j} + \psi_{t_{k}}^{\theta^{m}, j}$ \\
Compute $\Err_{j}^{\theta^{m}}(t_{k}) = \NRM{\Y^{\theta_{1}^{m}}(\overline{V}_{t_{k}}^{j}) + \phi_{t_{k}}^{\theta^{m}, j} -  \Y^{\theta_{1}^{m}}(v_{0}) }^{2}$.
}}
Compute $ L_{\loc}^{B}(\theta_{1}^{m}, \theta_{2}^{m}, \bar{\lambda}^{m}) = \frac{1}{B} \displaystyle\sum_{j=1}^{B} \PAR{\ABS{\Y^{1, \theta_{1}^{m}}(v_{0}) - y_{0}}^{2} + \sum_{k=1}^{N_{j}} \Err_{j}^{\theta^{m}}(t_{k})}$. \\
Update $\theta^{m+1} = \theta^{m} - \rho_{m} \nabla_{\theta} L_{\loc}^{B}(\theta^{m})$.
}
\caption{Locally additive eBSDE Algorithm - (LAeBSDE)}
\label{alg:locebsdeswitch}
\end{algorithm}

\begin{rmq}
\begin{enumerate}
\item We also tested a global solver, in the spirit of \cite{deepbroux2024}, based on the minimization of a terminal loss at the random horizon. In the switching setting, this approach requires to introduce trainable values for the unknown starting points of forward discretization, for $i \neq i_{0}$. Numerically, the method turned out to be unstable; the forward iteration generated large values in the loss, most likely due to the coupling term $G$. Clipping the differences inside $g$ at $C_{Y}$ did not improve stability, since it also suppressed the corresponding gradients. For this reason, we focused on the locally additive formulation.
\item  Algorithm \ref{alg:locebsdeswitch} could be extended to work for stochastic factor of dimension $d'>1$, by extending the definition of the random horizon $\tau$ to the return time in a centered ball around $v_{0}$. However, we expect the random horizon to have a greater mean, thus leading to high computational cost and possible time discretization error propagation. 
\end{enumerate}
\end{rmq}

\subsection{Deep Galerkin method for ergodic PDE}

Finally, we close this section by introducing an alternative residual-based neural-network algorithm for simulating systems of ergodic BSDEs, in the spirit of   the Deep Galerkin Method (DGM) introduced in \cite{sirignano2018dgm}. Since the DGM was introduced for high-dimensional parabolic PDE, a broad family of neural-network solvers for PDEs has developed around the same idea, of representing the solution by a neural network and training it to satisfy the equation. The DGM and the
closely related physics-informed neural networks of \cite{raissi2019pinn} both minimize the residual of the PDE, and differ in the sampling used. Extensions of the DGM to the standard HJB equations is presented in \cite{acl2022extensions}.

In the following, we extend the application of the Deep-Galerkin method to the system of ergodic PDE \eqref{PDEequation}, which we recall below 
\begin{equation*}
    \mathcal L y^i(v)+ F^{i}(v, \nabla y^{i}(v)\kappa)+\sum_{j \in \I} 
q^{i j}g(y^{j}(v)-y^{i}(v))=\lambda, \quad \forall \;  v\in \R^{d'},  \forall \; i \in \I. 
\end{equation*}
First, the equation is stationary and present an ergodic constant unknown $\lambda$, common to all equations, that is determined jointly with the solution $(y^{i}(.))_{i \in \I}$. Second, the equation is posed on the whole space $\R^{d'}$, so that there is no boundary condition and the choice of the sampling distribution for the residual becomes a modeling decision. While uniform sampling is common on bounded domain, we will prefer sampling from the invariant law of the underlying stochastic factor, in line with the unbounded ergodic setting. Instead, uniqueness is enforced through the normalization $y^{i_{0}}(v_{0}) = y_{0}$, in accordance with Theorem \ref{thm:exsebsde}. A key feature of the DGM is that the dimension $d'$ of the state variable $V^{v_0}$ does not affect the formulation of the method: no spatial grid is required, and the PDE residual is evaluated directly at randomly sampled points in $\mathbb{R}^{d'}$. The method approximates the unknown solution by a neural network and minimizes a loss function built from the PDE residual. This loss typically contains squared residuals evaluated at interior sampling points, together with a normalization term enforcing $y^{i_{0}}(v_{0}) = y_{0}$.

Let $\Y_{\theta}$ be a neural network defined on $\R^{d'}$ and valued in $\R^{I}$, with parameter $\theta$. For each $i \in \I$, define the residuals
\begin{eqnarray} \label{eq:residualsDGM}
R_{\theta}^{i}(v) = \Lr \Y_{\theta}^{i}(v) + F^{i}(v, \nabla \Y_{\theta}^{i}(v) \kappa) + \sum_{j \in \I} q^{ij} g(\Y_{\theta}^{j}(v) - \Y_{\theta}^{i}(v)) - \lambda_{\theta}.
\end{eqnarray}

\paragraph{Ergodic cost approximation.} The approximation of the ergodic cost $\lambda_\theta$ is based on its invariant-measure representation \eqref{eq:lambda_stat}. Replacing the exact solution by the neural network approximation, we define, for each regime $i$ in $\I$,
\begin{eqnarray} \label{eq:lambda_stat}
\lambda_{\theta}^{i} =  \int_{\mathbb R^{d'}}  \left(F^{i}(v, \nabla \Y_{\theta}^{i}(v) \kappa)+\sum_{j \in \I} 
q^{i j}g(\Y_{\theta}^{j}(v)-\Y_{\theta}^{i}(v))\right)\nu(dv).
\end{eqnarray}
where $\nu$ denotes the invariant distribution of the stochastic factor $V^{v_0}$. For the exact solution, the ergodic unknown $\lambda$ is common to all states and thus does not depend on the regime $i$. During the numerical approximation, however, the quantities $\lambda_{\theta}^{i}$ obtained from the network need not coincide. We therefore form a
common approximation $\lambda_{\theta}$ by averaging over the regimes,
\begin{eqnarray} \label{eq:mean_lambdai}
\lambda_{\theta} = \frac{1}{I} \sum_{i} \lambda_{\theta}^{i}.
\end{eqnarray}
An alternative approach would consist in treating $\lambda_\theta$ as an additional trainable scalar parameter, optimized jointly with the parameters of the neural network, as presented in the LAeBSDE and in \cite{deepbroux2024}. In this article, we will use the invariant measure representation \ref{eq:mean_lambdai} to approximate the ergodic cost in the DGM solver.

In this setting, direct sampling from the invariant measure is possible, as considered in \cite{gobet2024numerical}. In the general ergodic case, however, the invariant measure is typically not known explicitly and has to be approximated numerically. This can be done, for instance, using the recursive algorithm introduced in \cite{lamberton2002recursive}.

\paragraph{Loss function.} We now define the loss minimized by the DGM algorithm. Since Theorem \ref{thm:exsebsde} requires the uniform bound \eqref{eq:unifboundy},
we add a penalization term in order to encourage the neural approximation to remain in this admissible region, namely
\begin{eqnarray*}
\mathcal{P}^{i}_{\theta}(v) = \sum_{i \neq j} \max(\ABS{\Y_{\theta}^{i}(v) - \Y_{\theta}^{j}(v)} - C_{Y}, 0).
\end{eqnarray*}
 The loss function is then defined by
\begin{eqnarray}
L(\theta) = \e_{\nu} \SBRA{ \sum_{i\in\I} \ABS{R_{\theta}^{i}(v)}^{2} } + \ABS{\Y_{\theta}^{i_{0}}(v_{0}) - y_{0}}^{2} + \e_{\nu} \SBRA{ \sum_{i\in\I} \ABS{\mathcal{P}^{i}_{\theta}(v)}^{2}},
\end{eqnarray}
where $\e_{\nu}$ denotes the expectation with respect to the invariant distribution. In practice, the expectations entering the loss are approximated by empirical averages based on a set of batchsize $B$. This builds the DGM adaptation for ergodic BSDE, described in Algorithm \ref{alg:DGM}.

\begin{rmq}
Although we choose to sample the collocation points from the invariant distribution $\nu$, this choice is not mandatory for the residual term. One may alternatively sample points from another distribution, for instance uniformly on a compact region of interest, while still using invariant-measure sample for the approximation of $\lambda_{\theta}$. 
\end{rmq}

\begin{algorithm}[!ht]
\DontPrintSemicolon
Let $\Y^{\theta}$ be a neural network defined on $\R^{d'}$, valued in $\R^{I}$ with parameters $\theta$. Let $i_{0} \in \I$, $v_{0} \in \R^{d'}$ and $y_{0} \in \R$ be fixed, ensuring uniqueness of the Markovian solution to \eqref{eq:sebsdeform} such that $y^{i_{0}}(v_{0}) = y_{0}$.  \\
\For{$m=0$, ..., $M$}{
Sample $(\bar{V}_{k})_{1 \leq k \leq B}$ from the invariant distribution $\N \PAR{m, \int_{0}^{\infty} e^{-\mu s} \kappa \kappa^{\top} e^{- \mu^{\top}s}ds }$ \\
\For{$i=1, ..., B$}{Compute $\lambda_{\theta_{m}}^{i} =  \sum_{k = 1}^{B}  (F^{i}(\bar{V}_{k}, \nabla \Y_{\theta_{m}}^{i}(\bar{V}_{k}) \kappa)+\sum_{j \in \I} 
q^{i j}g(\Y_{\theta_{m}}^{j}(\bar{V}_{k})-\Y_{\theta_{m}}^{i}(\bar{V}_{k})).$
}
Compute $\lambda_{\theta_{m}} = \frac{1}{I} \sum_{i=1}^{I} \lambda^{i}_{\theta_{m}}$ \\
\For{$k=1, ..., B$}{
Set, $\phi_{t_{-1}} = 0$. \\
\For{$i = 1, ..., I$}{
$ R_{\theta^{m}}^{i}(\bar{V_{k}}) = \Lr \Y_{\theta_{m}}^{i}(\bar{V_{k}}) + F^{i}(\bar{V_{k}}, \nabla \Y_{\theta_{m}}^{i}(\bar{V_{k}}) \kappa) + \sum_{j \in \I} q^{ij} g(\Y_{\theta_{m}}^{j}(\bar{V_{k}}) - \Y_{\theta_{m}}^{i}(\bar{V_{k}})) - \lambda_{\theta_{m}},$ \\
$P^{i}_{\theta_{m}}(\bar{V}_{k}) = \sum_{i \neq j} \max(\ABS{\Y_{\theta}^{i}(v) - \Y_{\theta}^{j}(v)} - C_{Y}, 0)$.
}
}
Compute $ L_{DGM}^{B}(\theta_{m}) = \frac{1}{B} \displaystyle\sum_{k=1}^{B} \PAR{ \sum_{i=1}^{I} \ABS{R_{\theta^{m}}^{i}(\bar{V_{k}})}^{2} + \sum_{i=1}^{I} P^{i}_{\theta_{m}}(\bar{V}_{k}) + \ABS{\Y_{\theta_{m}}^{i_{0}}(v_{0}) - y_{0}}^{2}}$. \\
Update $\theta^{m+1} = \theta^{m} - \rho_{m} \nabla_{\theta} L_{DGM}^{B}(\theta^{m})$.
}
\caption{Deep Galerkin method for eBSDE - (DGM eBSDE)}
\label{alg:DGM}
\end{algorithm}

\section{Regime switching forward utilities in a stochastic factor model} \label{section:forwardutilitiesswitch}

In this section, we develop a general framework for decision-making in a regime-switching market using forward utilities, generalizing the results of \cite{hu2020systems}, where forward utilities are introduced in such a setting.

We start by recalling the financial market model introduced in \cite{hu2020systems}. Within this setting, we first derive a consistency stochastic partial differential equation that provides a broad characterization of regime-switching forward utilities. As a particular case, we recover the characterization of homothetic (power-type) forward utilities via systems of ergodic BSDEs obtained in \cite{hu2020systems}, and we derive systems of ergodic BSDEs associated with exponential and logarithmic forward utilities in regime switching markets.  These ergodic BSDE systems enable the numerical approximation of the corresponding forward utilities and optimal strategies.

\subsection{Regime switching stochastic factor model}

We consider an agent who can invest her wealth in one riskless asset and $n$ risky assets $S = (S^1, \dots , S^n)$, in a regime switching incomplete financial market. The regime switches are modeled by the continuous-time Markov chain (CTMC) $\alpha$ on the  state space $\I$ and transition rate matrix $\Qr = \PAR{q_{ij}}_{i, j \in \I}$, as defined in \eqref{eq:defalpha}.
In each regime market $i \in \I$, the risky assets are characterized by a market price of risk $(\theta^i(V_t^{v_0}))_{t\geq 0}$ and a volatility matrix $(\sigma^i(V_t^{v_0}))_{t\geq 0}$ driven by the $d'$-dimensional stochastic factor  $V^{v_0}$, as defined in \eqref{Vequation}.

The bond is assumed to be the numeraire, and hence the stock price dynamics discounted by the interest rate is given by 
\begin{eqnarray} \label{eq:stock_price}
dS_{t} = \diag(S_{t}) \sigma^{\alpha_{t^{-}}}(V_{t}^{v_0}) \PAR{\theta^{\alpha_{t^-}}(V_{t}^{v_0})dt + dW_{t}},
\end{eqnarray}
 
\begin{ass} \label{asslispch_switch}
\begin{enumerate}
            \item For all $v \in \R ^{d'}$ and $i \in \I$, the $(n,d)$ matrix  $\sigma^i(v)$ has full row rank $n$. 
            \item For all market regime $i \in\I$, the market price of risk $\theta^{i}:\R^{d'} \to \R^n$ is uniformly bounded and Lipschitz continuous. 
\end{enumerate}
\end{ass}

The agent invests a proportion $\bar{\pi} = \PAR{\bar{\pi}^{1}, ..., \bar{\pi}^{n}}^{\top}$ of her wealth $X^{\pi}$ in the $n$ risky assets. For an initial value $X_{0}^{\pi} = x_{0} \in \R^+$, assuming the self-financing condition holds and rescaling the strategy vector by the volatility, the wealth process $X$ evolves as 
\begin{eqnarray} \label{wealth_switch}
dX_{t}^{\pi} = X_{t}^{\pi} \pi_{t} \cdot \PAR{\theta^{\alpha_{t}}(V_{t}^{v_0})dt + dW_{t}}, \quad \pi_t = \sigma^{\alpha_{t^-}}(V_t^{v_0})^{\top} \bar \pi_t \in \R^d . 
\end{eqnarray}
Let $\PAR{\Pi^{i}}_{i \in \I}$ be a family of closed convex subsets of $\R^{d}$, representing the constraint sets associated with each market regime. For any $i \in \I$, we assume that $\pi^{i} = \sigma^{i}(V_{t}^{v_0})^{\top} \bar{\pi_{t}}^{i} \in \Rr^{i}$, with  $\Rr^{i}$ a vector space of $\R^d$. In the following, for any $a \in \R^{d}$, $a^{\Rr^{i}}$ denotes the orthogonal projection on $\Rr^{i}$ and $a^{i, \perp}$ its orthogonal projection onto $\Rr^{i}$. For any $t \geq 0$, the set of admissible strategies in $[0,t]$ is defined as
\begin{multline}
\A_{t} = \left\{\pi_{s} = \sum_{i \in \I}  \pi_{s}^{i} \ind_{\{\alpha_{s^-}=i\}}, \, \text{for } s \in \SBRA{0, t},  ; \;  \forall \,  i \in \I, \, \pi_{s}^{i} \in \Rr^{i}, \right.\\
\left.  \pi^{i}\,\text{is} \, \f \, \text{prog measurable and } \int_{0}^{t} \ABS{\pi_{s}^{i}}^{2}ds < \infty, \, \p-\text{a.s}.  \right\}.
\end{multline}
For each $t\geq 0$, the predictable strategy $(\pi_t)_{t \geq 0}$ is assumed to lie in $\Rr^{\alpha_{t^{-}}}$, which depends on the current market regime. 
Finally, the set of admissible strategies is denoted by $\mathcal{A} = \bigcup_{t\geq 0 } \mathcal{A}_t$.

\subsection{Forward utilities in a regime-switching environment} \label{sect:sebsdeut}

In this section, we give a general characterization of  forward utilities in a regime-switching market. The introduction of regular random field spaces for the study of differentiability of Itô random fields is recalled in Annex \ref{annex:spaces}, see also \cite{nicole2013exact}.
The decision criteria are built upon a family of random utilities $\PAR{U^{i}}_{i \in \I}$, each modeling the agent's preferences in regime $i$. 
For all $i \in  \I$, the utility random field $U^i :  \R^+ \times \R^{+} \times \Omega \to \R$  associated with market  
 regime $i$ is an $\f$-adapted random field such that:
\begin{itemize}
\item The functions $x \in \R^{+} \mapsto U^i(t,x,\omega)$  are  nonnegative,  strictly concave increasing  functions  of class $\mathcal C^{2}$ on $]0, \infty[$, $(\omega,t)$ a.s.
\item $u_0^i:= U^i(0, \cdot)$ is a standard (deterministic) utility function.
\item  Inada conditions: $\lim_{x \to 0^{+}} U^{i}(t, x) = 0, \quad \lim_{x \to 0^{+}} U_{x}^{i}(t, x) = +\infty, \quad \lim_{x \to +\infty} U_{x}^{i}(t, x) = 0 \quad \text{a.s.}$
\end{itemize}

More specifically, we consider regime specific utilities  $\PAR{U^{i}(t, x)}_{i \in \I}$that are Itô random fields with local characteristics $\PAR{\beta^{i}, \gamma^{i}} \in \K^{ 3, \epsilon}_{\loc} \times \overline{\K}^{ 3, \epsilon}_{\loc}$ with $\epsilon > 0$, namely for $i\in \I$
\begin{eqnarray}
    dU^{i}(t, x) = \beta^{i}(t, x)dt + \gamma^{i}(t, x) \cdot dW_{t}.
\end{eqnarray}
By Theorem $2.2$ in \cite{nicole2013exact}, $U^{i}$ is a $\K^{ 3, \epsilon'}_{loc}$ semimartingale, for any $\epsilon' < \epsilon$.

\paragraph{Regime switching utility} The regime-switching  utility $U$ is the $\g$-adapted càdlàg process defined as follows:
\begin{eqnarray} \label{eq:Usum}
U(t, x) = U^{\alpha_t}(t,x) = \sum_{i \in \I} U^i(t,x) \ind_{\{\alpha_t = i\}}.
\end{eqnarray}
At each regime-switching time $T_k$ (i.e jump times of $\alpha$), the  random field $U$ jumps from the utility $U^{\alpha_{T_{k-1}}}$ in market regime $\alpha_{T_{k-1}}$ to $U^{\alpha_{T_{k}}}$ in market regime $\alpha_{T_{k}}$. The regime-switching utility $U$ \eqref{eq:Usum} have thus the following dynamics: 
\begin{multline} \label{eq:genitoU}
U(t, x) = U^{\alpha_{0}}(0, x) + \sum_{i \in \I} \int_{0}^{t} \ind_{\{\alpha_{s^-}=i\}}\left( \beta^{i}(s, x) +  \sum_{j \in \I}  q^{ij}(U^j(s,x) - U^i(s,x))\right) ds  \\
+ \int_{0}^{t} \ind_{\{\alpha_{s^-}=i\}} \gamma^{i}(s, x) \cdot dW_s +  \int_0^t \int_{\I^2} \left( U^j(s,x) - U^i(s,x)\right)\ind_{\{\alpha_{s^-} =i\}}\tilde  N(ds, di,dj),
\end{multline}
with $\tilde  N$ the compensated Poisson measure as defined in Lemma \ref{lemma:eBDSEjump}.

The decision criterion maintains time consistency within the given investment context, in the sense of the following definition, as introduced in \cite{hu2020systems}. The optimal strategy provides maximal satisfaction to the agent, which is preserved at all times in the future. 

\begin{defi}[\textit{Regime switching forward utility}] \label{def:switchut}
A family of utility random fields $\PAR{U^{i}(t, x)}_{i \in I}$ generates a regime switching forward utility if the $\g$-adapted càdlàg process $U$ defined by \eqref{eq:Usum} satisfies the time consistency property:
\begin{itemize}
    \item  For any admissible strategy $\pi \in \mathcal{A}$, $U(t, X_{t}^{\pi})$, $t\geq 0$, is a locale supermartingale, i.e.
    \begin{eqnarray*}
        U(t, X_{t}^{\pi}) \geq \e \SBRA{U(s, X_{s}^{\pi}) | \Gr_{t}}.
    \end{eqnarray*}
    \item There exists an admissible optimal strategy $\pi^{*} \in \mathcal{A}$ such that $U(t, X_{t}^{\pi^{*}})$, $t\geq 0$, is a local martingale, i.e
        \begin{eqnarray*}
        U(t, X_{t}^{\pi^{*}}) = \e \SBRA{U(s, X_{s}^{\pi^{*}}) | \Gr_{t}}.
    \end{eqnarray*}
\end{itemize}
\end{defi}

\paragraph{Time consistency} In Theorem \ref{thm:consistswitch} we give  a sufficient  characterization of time consistency for  regime-switching forward utilities.  A general framework for forward utilities with jumps and the related nonlinear consistency SPDE has been introduced \cite{matoussi2022dynamic}, extending the results of \cite{musiela2010stochastic}, \cite{nicole2013exact} in the continuous case, where the authors rely on a generalization of Itô-Ventzel's formula with jumps introduced in \cite{oksendal2007ito}. 
Here, the regime-switching setting can be dealt with directly. Indeed, despite the jumps of the market price of risk $\theta^{i}$, the wealth process $X^{\pi}$ is a continuous-time diffusion process. 
 Hence, the dynamics of  the compound process $U(t, X_{t}^{\pi})$ is first obtained by applying the Itô-Ventzel's formula with jumps. 
For $U$ to satisfy the martingale consistency property from Definition \ref{def:switchut}, a necessary condition is that its drift $\beta^{U}(t, X_{t}^{\pi}) \leq 0$, with equality holding along some optimal strategy. A natural candidate optimal strategy $\pi_{t}^{*}(x)$ is thus the one that maximizes $\beta^{U}(t,\cdot)$, when seen as function of $\pi$. It is actually sufficient to impose that this maximum equals $0$ to ensure that $\beta^{U}(t,  X_{t}^{\pi}) \leq 0$ for all admissible strategies $\pi$, thus providing time consistency of the regime-switching forward utility $U$. Finally, the following assumption ensures the admissibility of the optimal strategy.
\begin{ass} \label{ass:MM_adm}
Let $U$ be a regime-switching utility as defined in \eqref{eq:genitoU}. There exists a process $K\in L^2(dt)$ such that a.s., for all $i\in\I$, $t\ge0$, $x>0$,
\begin{eqnarray}
\NRM{\gamma^{i,\perp}_x(t,x)} \le K_t\, U^i_x(t,x),
\qquad
\NRM{\gamma^{i,\perp}_{xx}(t,x)} \le K_t\, \ABS{U^i_{xx}(t,x)},
\end{eqnarray}
where $\gamma^{i,\perp}=\Proj_{(\Rr^i)^\perp}\gamma^i$.
\end{ass}

\begin{thm} \label{thm:consistswitch}
Let $U$ be a regime-switching utility as defined in \eqref{eq:genitoU}, verifying Assumption \ref{ass:MM_adm}. Assume that for each $i \in \I$, the utility random field $U^i$ of local characteristics $(\beta^i, \gamma^i)$ verifies: 
\begin{align}
\beta^{i}(t, x) &= - \frac{1}{2}U_{xx}^{i}(t, x) \underset{\pi_{t} \in \Pi}{\inf} \Qr^i(t, X_{t}^{\pi}, \pi)
- \sum_{j \in \I} \SBRA{U^{j}(t, x) - U^{i}(t, x)}q^{ij}, \label{eq:driftcondswitch} \\
\text{where} \quad &\Qr^{i}(t, x, \pi) = \NRM{x \pi}^{2} + 2x \pi \cdot \frac{U_{x}^{i}(t, x) \theta^{i}(V_{t}) + \gamma_{x}^{i}(t, x)}{U_{xx}^{i}(t, x)}. \label{eq:defQswitch}
\end{align}
Then, the regime switching utility $U$ defined in \eqref{eq:Usum} is a regime switching forward utility in the sense of Definition \ref{def:switchut}. 
The optimal investment strategy $\pi_{t}^{*}$ is given by
\begin{equation} \label{eq:optportgenswitch}
    \pi_{t}^{*}= \sum_{i \in \I}  \pi_t^{i,*}  \ind_{\{\alpha_t = i\}}, 
\end{equation}
where  $\pi^{i,*}_t :=\pi^{i,*}_t(X_t^{\pi^*})$ is the optimal strategy in regime $i$,  with 
\begin{equation}
\label{eq:pi_regime_i}
  \pi^{i,*}_t(x) =  \frac{- 1}{x U_{xx}^{i}(t, x)} \Proj_{\Rr^{i}}  \PAR{U_{x}^{i}(t, x) \theta^{i} (V_{t}) + \gamma_{x}^{i}(t, x)}. 
\end{equation}
In particular,
\begin{equation*}
    Q^i (t,x,\pi^{i,*}_t)= - \NRM{x \pi^{i,*}_t}^{2}.
\end{equation*}
\end{thm}

\begin{rmq}
Note that the expression for the optimal portfolio in regime $i$ coincides with the non-switching case, see e.g.  \cite{nicole2013exact}. However, the dynamics of the regime-switching optimal strategy differ from those in the single regime setting since the consistency condition for each regime \eqref{eq:driftcondswitch} induce a coupling of the dynamics of regime-specifc utilities .
\end{rmq}

\begin{preuve}
Let $(\pi_t)_{t\geq 0} = (\sum_{i\in \I} \pi^{i}_t \ind_{\{\alpha_{t^-} = i\}})_{t\geq 0} \in \mathcal{A}$ be an admissible strategy, associated with the wealth $X^\pi$.  Under regularity assumptions on the local characteristics, we can apply the It\^o-Ventzel's formula with jumps to  $U(t, X_{t}^{\pi})$. By \eqref{eq:genitoU}, this yields that:
\begin{align*}
dU(t, X_{t}^{\pi}) &= \sum_{i\in \I} \ind_{\{\alpha_{t^-} =i \}} \left(\beta^{i}(t, X_t^\pi) +  \sum_{j \in \I}  q^{ij}(U^j(t, X_t^\pi) - U^i(t, X_t^\pi)) dt \right) \\
&\quad +  \sum_{i\in \I} \ind_{\{\alpha_{t^-}=i \}}\gamma^{i}(t, X_t^\pi) \cdot dW_t + \int_{\I^2} \left( U^j(t,X_{t}^{\pi}) - U^i(t,X_{t}^{\pi})\right)\ind_{\{\alpha_{t^-} =i\}}\tilde  N(dt, di,dj) \\
&\quad  + \sum_{i\in \I} \ind_{\{\alpha_{t^-}=i \}}  \left( U_{x}^i(t, X_{t}^\pi) \SBRA{X_{t}^{\pi} \pi_{t}^i \cdot  \PAR{\theta^{i}(V_{t})dt + dW_{t}}} + \frac{1}{2}U_{xx}^i(t, X_{t}) \NRM{X_{t}^{\pi} \pi_{t}^i}^{2}dt  + X_{t}^{\pi} \gamma_{x}^i (t, X_{t}^{\pi})  \cdot \pi_{t}^i dt \right).
\end{align*}
Note that there is no jump quadratic variation because the wealth process $(X_t)_{t \geq 0}$ is continuous. In integral form, the Brownian terms and compensated terms of the jump part  are local martingales. Hence, a sufficient condition for consistency in this framework is that the drift of $U(t, X_{t}^{\pi})$ is negative for any admissible strategy, and that there exists an optimal strategy $\pi^{*}$ along which it vanishes. Accounting for the jumps' contribution, the drift of $U(t, X_{t}^{\pi})$ takes the form
\begin{align}
\beta^{U}(t, X_t^\pi) &=  \sum_{i\in \I} \ind_{\{\alpha_t =i \}} \left( \beta^{i}(t, X_{t}^{\pi}) + \frac{1}{2}U_{xx}^i(t, X_{t}^{\pi}) \Qr^i(t, X_{t}^{\pi}, \pi^i_t)
+ \sum_{j \in \I} \SBRA{U^{j}(t, x) - U^{i}(t, x)}q^{i j}\right), \label{eq:driftsiwtchpreuve}
\end{align}
where $\Qr^i$ is given by \eqref{eq:defQswitch}. For $i \in \I$, assume that the solution $U^i(t, .)$ obtained with \eqref{eq:driftcondswitch} is concave. Then,  its second derivative is negative, and thus  each term of the sum in  \eqref{eq:driftsiwtchpreuve} is maximal when $Q^i$ is minimal. The candidate optimal policy in regime $i$ is then defined by $\pi_{t}^{i,*} = \pi_{t}^{i,*}(X^{\pi^*}_t)$, with  $\pi_{t}^{i,*}(x):=  \underset{\pi \in \Rr^i}{\argmin} Q^i(t, x, \pi)$,  and from the first order condition,
\begin{eqnarray*}
\pi_{t}^{i,*}(x) = \Proj_{\Rr^{i}} \PAR{\frac{- 1}{x U_{xx}^i(t, x)} \PAR{U_{x}^i(t, x) \theta^{i}(V_{t}) + \gamma_{x}^{i}(t, x)}}.
\end{eqnarray*}
For all $i\in \I$ the quadratic form $\Qr^i$ attains its minimal value $\Qr(t, X_{t}^{\pi^*}, \pi^{i,*}) = - \NRM{X_{t}^{\pi^*} \pi_{t}^{i,*}}^{2}$ for this strategy. The drift of $U(t, X_t^\pi)$ thus satisfies for any admissible strategy $\pi \in \A$
\begin{align*}
\beta^{U}(t, X_t^\pi) &=\sum_{i\in \I} \ind_{\{\alpha_t =i \}} \left( \beta^{i}(t, X_{t}^{\pi}) + \frac{1}{2}U_{xx}^i(t, X_{t}^{\pi}) \Qr^i(t, X_{t}^{\pi}, \pi^i_t)
+ \sum_{j \in \I} \SBRA{U^{j}(t, x) - U^{i}(t, x)}q^{i j}\right) \\
&\leq \sum_{i\in \I} \ind_{\{\alpha_t =i \}} \left( \beta^{i}(t, X_{t}^{\pi}) - \frac{1}{2}U_{xx}^i(t, X_{t}^{\pi})  \NRM{X_{t}^{\pi^*} \pi_{t}^{i,*}}^{2}
+ \sum_{j \in I} \SBRA{U^{j}(t, X_{t}^{\pi^*}) - U^{i}(t, X_{t}^{\pi^*})}q^{i j}\right).
\end{align*}
By the assumption \eqref{eq:driftcondswitch}, the right hand side of the inequality is equal to 0 and thus the process $(U(t, X_t^\pi))_{t\geq 0}$ is a supermartingale. \\
It remains to show that for all $i \in \I$,  $U^i$ is an increasing and concave positive random field and the $\pi^* \in \A$. This can be obtained  as  our framework falls under the framework of forward utilities with jumps introduced in \cite{matoussi2022dynamic},  with the Lévy measure replaced by the jump measure of the CTMC $\alpha$ evolving on the finite state space $\I$. Under the Inada condition and by  Assumption \ref{ass:MM_adm}, Theorem 3.8 in \cite{matoussi2022dynamic} is verified, which allows us to conclude. 

\end{preuve}

\paragraph{Regime switching power forward utilities}

A typical choice for the regime specific utility random fields $(U^i)_{i \in \I}$, are homothetic power type utilities, of separable form 
\begin{align} \label{eq:powUi}
U^i(t,x) = P^i_t \frac{x^{\delta_i}}{\delta_i}, \quad \delta_i \in (-\infty,0) \cup (0,1), 
\end{align}
and with $P^i$ an Itô diffusion process with the following dynamics:
\begin{equation*}
    dP^i_t = P^i_t (b_t^i dt + \nu_t^i \cdot dW_t). 
\end{equation*}
The following proposition is a corollary of Theorem \ref{thm:consistswitch}, and shows that  regime-switching power forward utilities have necessarily the same risk aversion in each regime.
\begin{prop}
    \label{prop:consistpower}
    Let $(U^i)_{i \in \I}$ be a family of power type utilities as above, defining a regime switching forward utility $U$ by  \eqref{eq:Usum}.  Then,  under the assumptions of Theorem \ref{thm:consistswitch}, the regime specific utilities have necessarily the same risk aversion coefficient: 
    \begin{equation*}
        \delta_i = \delta, \quad \forall \; i\in \I, 
    \end{equation*}
   and  the consistency assumption \eqref{eq:driftcondswitch} of Theorem can be written as: 
   \begin{equation*}
      b_t^i  = \frac{\delta_i(\delta_i - 1)}{2} \Proj_{\Rr^i} \PAR{\frac{  \nu^i_t + \theta^{i}(V_t)}{1 - \delta_i}}^2 - \sum_{j\in \I}\frac{P^j_t}{P^i_t} q^{ij}. 
\end{equation*}
\end{prop}

\subsection{Link between homothetic switching forward utilities and systems of ergodic BSDEs} \label{sect:sebsde_ut}

A straightforward application of Proposition \ref{prop:consistpower} and Itô's formula allows us to establish a correspondence between regime-switching forward utility in power form and a system of ergodic BSDEs, thus recovering  the characterization of \cite{hu2020systems}. 

\begin{coro}[\cite{hu2020systems}] \label{prop:powerswitchut}
Consider a family of utility random fields $\PAR{U^{i}(t, x)}_{i \in \I}$ of homothetic power type 
\begin{eqnarray} \label{eq:powerswitchut}
U^{i}(t, x) = \frac{x^{\delta}}{\delta} e^{y^{i}(V_{t}) - \lambda t}, \quad \delta \in (-\infty,0) \cup (0,1), 
\end{eqnarray}
where $\PAR{(y^{i}(.), z^{i}(.))_{i \in \I}, \lambda}$ is a Markovian solution of the system of ergodic BSDE \eqref{eq:sebsdeform}, with generator $F^{i}$ and coupling term $G^{i}$ given by 
\begin{eqnarray}
    F^{i}(v, z) &=& \frac{1}{2}\delta (\delta - 1) \dist^{2} \PAR{\Rr^i, \frac{z + \theta^{i}(v)}{1 - \delta}} + \frac{\delta}{2(1 - \delta)} \NRM{z + \theta^{i}(v)}^{2} + \frac{\NRM{z}^{2}}{2}, \label{eq:powergenswitch} \\
    G^{i}(y) &=& \sum_{j \in \I} \PAR{e^{y^{j} - y^{i}} - 1}q^{ij}.
\end{eqnarray}
This family generates a Markovian regime switching forward performance process $U$, and the associated optimal strategy in regime $i$ takes the form
\begin{eqnarray} \label{eq:optport_power_switch}
    \pi_{t}^{i, *} = \Proj_{\Rr^{i}} \PAR{\frac{z^{i}(V_{t}) + \theta^{i}(V_{t})}{1 - \delta}}.
\end{eqnarray}
\end{coro}

Similary, Theorem \ref{thm:consistswitch} also allows us to identify systems of ergodic BSDEs associated with regime-switching forward utilities of exponential and logarithmic types. For switching exponential forward utilities, it is more convenient to use  the discounted amount of wealth invested in the stock $\alpha_{t}  = X_t^\pi  \pi_t $ as a control variable, leading to the following wealth process dynamics:
\begin{eqnarray}
dX_{t}^{\alpha} = \alpha_{t}^{\top} \PAR{\theta^{\alpha_{t^-}}(V_{t})dt + dW_{t}}.
\end{eqnarray}

\begin{prop}[Exponential forward utility in regime-switching market] \label{prop:switchutexp}
Consider a family of utility random fields $\PAR{U^{i}(t, x)}_{i \in I}$ of homothetic exponential type 
\begin{eqnarray} \label{eq:expswitchut}
U^{i}(t, x) = - e^{- \gamma x} e^{y^{i}(V_{t}) - \lambda t}, \quad \gamma \in \R
\end{eqnarray}
where $\PAR{(y^{i}(.), z^{i}(.))_{i \in \I}, \lambda}$ is a Markovian solution of the system of ergodic BSDE \eqref{eq:sebsdeform}, with generator $F^{i}$ and coupling term $G^{i}$ given by 
\begin{eqnarray}
    F^{i}(v, z) &=& \frac{\gamma^{2}}{2} \dist^{2} \PAR{\Rr^i, \frac{z + \theta^{i}(v)}{\gamma}} - \frac{1}{2} \NRM{z + \theta^{i}(v)}^{2} + \frac{\NRM{z}^{2}}{2}, \label{eq:driverexpswitch} \\
    G^{i}(y) &=& \sum_{j \in \I} \PAR{e^{y^{j} - y^{i}} - 1}q^{ij} \label{eq:couplexpswitch}
\end{eqnarray}
This family generates a Markovian regime switching forward performance process $U$, and the associated optimal strategy in regime $i$ takes the form
\begin{eqnarray}
    \alpha_{t}^{i, *} = \Proj_{\Rr^{i}} \PAR{\frac{z^{i}(V_{t}) + \theta^{i}(V_{t})}{\gamma}}.
\end{eqnarray}
\end{prop}

\begin{preuve}
The proof is a straightforward application of It\^o's formula and Theorem \ref{thm:consistswitch}.
\end{preuve}

The generator \eqref{eq:driverexpswitch} coincides with the one associated to homothetic exponential forward utility from \cite{liang2017representation}. We also recover the same exponential coupling term \eqref{eq:couplexpswitch} as for regime switching forward performance process in power form. Finally, we state a similar result for utilities of logarithmic type.

\begin{prop}[Logarithmic forward utility in regime-switching market]
Consider a family of utility random fields $\PAR{U^{i}(t, x)}_{i \in I}$ of homothetic logarithmic type 
\begin{eqnarray} \label{eq:logswitchut}
U^{i}(t, x) = \ln(x) + y^{i}(V_{t}) - \lambda t,
\end{eqnarray}
where $\PAR{(y^{i}(.), z^{i}(.))_{i \in \I}, \lambda}$ is a Markovian solution of the system of ergodic BSDE \eqref{eq:sebsdeform}, with generator $F^{i}$ and coupling term $G^{i}$ given by 
\begin{eqnarray}
    F^{i}(v) &=& - \frac{1}{2} \dist^{2} \PAR{\Rr^i, \theta^{i}(v)} + \frac{1}{2} \NRM{\theta^{i}(v)}^{2}, \label{eq:driverlogswitch} \\
    G^{i}(y) &=& \sum_{j \in \I} \PAR{y^{j} - y^{i}}q^{ij}. \label{eq:coupllogswitch}
\end{eqnarray}
This family generates a Markovian regime switching forward performance process $U$, and the associated optimal strategy in regime $i$ takes the form
\begin{eqnarray}
    \alpha_{t}^{i, *} = \Proj_{\Rr^{i}} \PAR{\theta^{i}(V_{t})}.
\end{eqnarray}
\end{prop}

The generator \eqref{eq:driverlogswitch} also coincides with the one from \cite{liang2017representation} in the single state setting. The coupling term \eqref{eq:coupllogswitch} however is not exponential, because of the additive form of the logarithmic utility \eqref{eq:logswitchut}.

\section{Numerical results} \label{section:numresswitch}

In this section, we present numerical experiments illustrating the performance of Algorithms \ref{alg:locebsdeswitch} (LAeBSDE) and \ref{alg:DGM} (DGM) for the approximation of the Markovian solution $\PAR{(y^{i}(.), z^{i}(.))_{i \in \I}, \lambda}$ of the system of ergodic BSDEs \eqref{eq:rdtBSDEswitch}, with a particular emphasis on systems arising from regime-switching utilities of power type, see Corollary \ref{prop:powerswitchut}.

\paragraph{Common numerical setting.} Numerical experiments are conducted using Intel(R) Xeon(R) CPU @ 2.20GHz with 25GB of RAM. Unless otherwise specified, the stochastic factor $V$ is an Ornstein-Uhlenbeck (OU) process
\begin{eqnarray} \label{eq:OUswitch}
dV_{t} = \mu \PAR{m - V_{t}} dt + \kappa^{\top}dW_{t}, \quad V_{0} = v_{0},
\end{eqnarray}
with $m \in \R^{d'}$, $\mu \in \R^{d' \times d'}$ and $\kappa \in \R^{d' \times d}$, which satisfies the dissipativity Assumption \ref{ass:weakdissass}. The invariant measure is Gaussian, namely 
\begin{eqnarray}
\N \PAR{m, \int_{0}^{\infty} e^{-\mu s} \kappa \kappa^{\top} e^{-\mu^{\top}s}ds}.
\end{eqnarray}
Points for the DGM and for validation error are sampled from this invariant distribution. Both solvers use feedforward neural networks with two hidden layers of $20 + Id$ neurons and hyperbolic tangent activations. We checked that increasing the depth and width does not improve the accuracy in our tests. The networks are trained by stochastic gradient descent using Adam optimize with initial learning rate $7 \times 10^{-4}$, batch size $B=100$, and $10000$ gradient descent steps. The algorithms are implemented in Python with \textit{Tensorflow} library. The code of both solvers is available on github : \url{https://github.com/gubrx/DeepSolvers-Ergodic-BSDE-Systems}.

\subsection{Some explicitly solvable examples}  \label{sect:4_explicit}

We first construct a family of systems of ergodic BSDEs admitting explicit Markovian solutions, which will serve as benchmarks for both algorithms. The construction starts from a family of ansatz $(y^{i})_{i \in \I}$ satisfying the condition of uniqueness of Theorem \ref{thm:exsebsde}, and a prescribed ergodic cost $\lambda$. We then define the market price of risk so that the ergodic PDE system \eqref{PDEequation} is satisfied. We consider the unconstrained power-type generator of Proposition \ref{prop:powerswitchut}, namely $\Pi = \R^{d}$ and
\begin{eqnarray} \label{eq:gensect4rappel}
    F^{i}(v, z) = \frac{\delta}{2(1 - \delta)} \NRM{z + \theta^{i}(v)}^{2} + \frac{\NRM{z}^{2}}{2}, \quad
    G^{i}(y) = \sum_{j \in I} \PAR{e^{y^{j} - y^{i}} - 1}q^{ij},
\end{eqnarray}
where $\delta\in(0,1)$. 

\begin{prop} \label{prop:exemplegen}
Let $(y^{i}(.))_{i \in \I}$ be a family of $C^{2}$ functions such that, for all $i \in \I$:
\begin{itemize}
\item $y^{i}$ has a bounded first derivative,
\item $\ABS{y^{i} - y^{j}}$ is bounded uniformly in $i$,
\item $\Lr y^{i}$ is bounded.
\end{itemize}
Set $z^i(v)=\kappa \nabla y^i(v) \in \R^d$ and let $\lambda\in\mathbb R$ be such that
\begin{eqnarray} \label{eq:lambda_cond_ex}
\lambda > \sup_{i\in I,\,v} \SBRA{\Lr y^i(v)+G^i(y(v))+\frac{1}{2}\NRM{z^i(v)}^2}.
\end{eqnarray}
Let $\mathbf{1} \in \R^d$ denote the unit vector and define
\begin{eqnarray}
\label{eq:mprexplsol}
\theta^i(v) = -z^i(v)+ \mathbf{1} \sqrt{\frac{2(1-\delta)}{\delta}\PAR{\lambda-\Lr y^i(v)-G^i(y(v))-\frac12\NRM{z^i(v)}^2}}.
\end{eqnarray}
Then $\PAR{(y^{i}(.),z^{i}(.))_{i \in \I}, \lambda}$ is the unique Markovian solution of the systems of ergodic BSDEs \eqref{eq:sebsdeform} with generator and coupling term given by \eqref{eq:gensect4rappel} and market price of risk given by \eqref{eq:mprexplsol}.
\end{prop}

In the following numerical experiments, the factor process is the Ornstein-Uhlenbeck process \eqref{eq:OUswitch} with $V$ valued in $\R^{d'}$ and $W$ is a $d$-dimensional Brownian motion. We consider benchmark solutions of the form
\begin{eqnarray} \label{eq:explicit_yi}
y^i(v)=\psi(v)+c_i+a_i\phi(v), \quad c_i,a_i\in\mathbb R
\end{eqnarray}
where $\psi$ is sublinear and $\phi$ is bounded and smooth with bounded first and second derivatives. The differences $y^{i} - y^{j}$ are then explicit bounded functions of $\phi$ and the coupling term is given by
\begin{eqnarray}
G^i(y(v)) = \sum_{j\in I}\PAR{\exp\PAR{c_j-c_i+(a_j - a_i)\phi(v)}-1}q^{ij}.
\end{eqnarray}
We fix a direction $\ell\in\mathbb R^{d'}$ and consider functions of the projection $\ell\cdot v$. 

\begin{exe}
\label{exe:explicit_profiles}
Let $w>0$, $\ell\in\mathbb R^{d'}$. We consider the two profiles
\begin{enumerate}
\item[(T)] \emph{Hyperbolic tangent}
\[
\phi_T(v)=\tanh(w\,\ell\cdot v).
\]
\item[(R)] \emph{Rational}
\[
\phi_R(v)=\frac{\ell\cdot v}{\sqrt{1+w^2(\ell\cdot v)^2}}.
\]
\end{enumerate}
with $w > 0$ and $\psi \equiv 0$ in the sequel, so that Proposition \ref{prop:exemplegen} can be applied.
\end{exe}
Since the training losses of the two methods correspond to different objectives, they are not directly comparable. We therefore consider $L^2$ validation errors computed under the invariant measure $\nu$ of the stochastic factor. Given approximations $\bar{y}, \, \bar{z}$, and $(v_{1}, \dots, v_{M})$ an i.i.d sample from $\nu$, we define
\begin{eqnarray}
    \E_{L^{2}(\nu)}(\bar{y}, M) = \frac{1}{MI} \sum_{m=1}^{M} \sum_{i=1}^{I} \ABS{y^{i}(v_{m}) - \bar{y}^{i}(v_{m})}^{2}, \quad \E_{L^{2}(\nu)}(\bar{z}, M) = \frac{1}{MI} \sum_{m=1}^{M} \sum_{i=1}^{I} \NRM{z^{i}(v_{m}) - \bar{z}^{i}(v_{m})}^{2},
\end{eqnarray}
with $M=1000$ in all reported experiments.

\paragraph{Hyperbolic tangent example (T) -} We first consider a two-regime example ($I=2$) with a one-dimensional OU process ($d'=1$) initialized at $v_{0} = 0$, with mean-reversion coefficient $\mu=2$, asymptotic mean $m=0$ and volatility coefficient $\kappa = 0.65$.  The transition rates are $q_{12} = 0.4$, $q_{21} = 0.8$, and the reference solution is
\begin{eqnarray*}
    y^{1}(v) = 1 - 0.3 \tanh(0.8v), \quad y^{2}(v) = 1 + 0.3 \tanh(0.8 v).
\end{eqnarray*}
normalized by $y^{1}(0) = 1.$ The ergodic cost is set to $\lambda = 0.811$, which satisfies condition \eqref{eq:lambda_cond_ex}. For the LAeBSDE algorithm, the discretization parameters are $h=0.01$ and the minimal horizon $T_{0} = 1$.

\begin{figure}[!ht]
\centering
\begin{subfigure}[t]{0.485\textwidth}
\centering
\includegraphics[width=\linewidth]{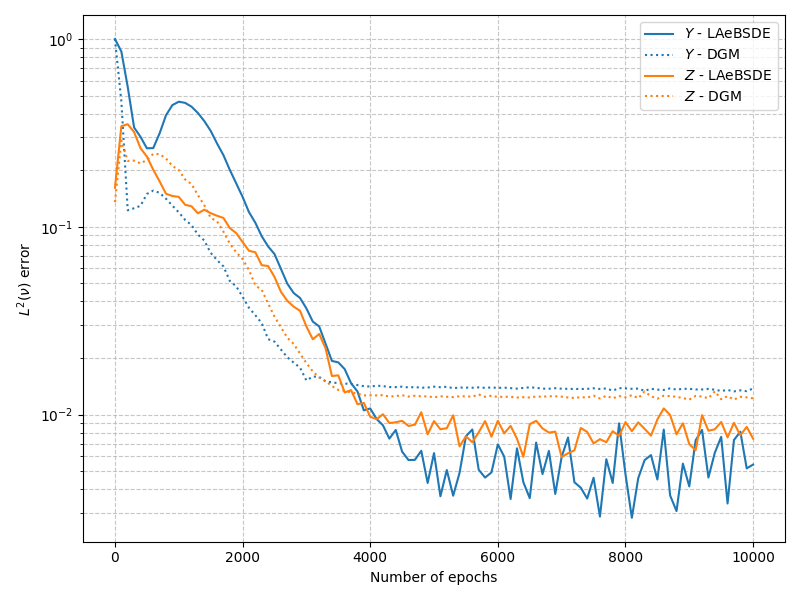}
\caption{Convergence of $\mathcal E_{L^2(\nu)}(\bar y,1000)$ and $\mathcal E_{L^2(\nu)}(\bar z,1000)$ as a function of the number of epochs.}
\label{fig:tanh_l2_conv}
\end{subfigure}\hspace{0.02\textwidth}
\begin{subfigure}[t]{0.485\textwidth}
\centering
\includegraphics[width=\linewidth]{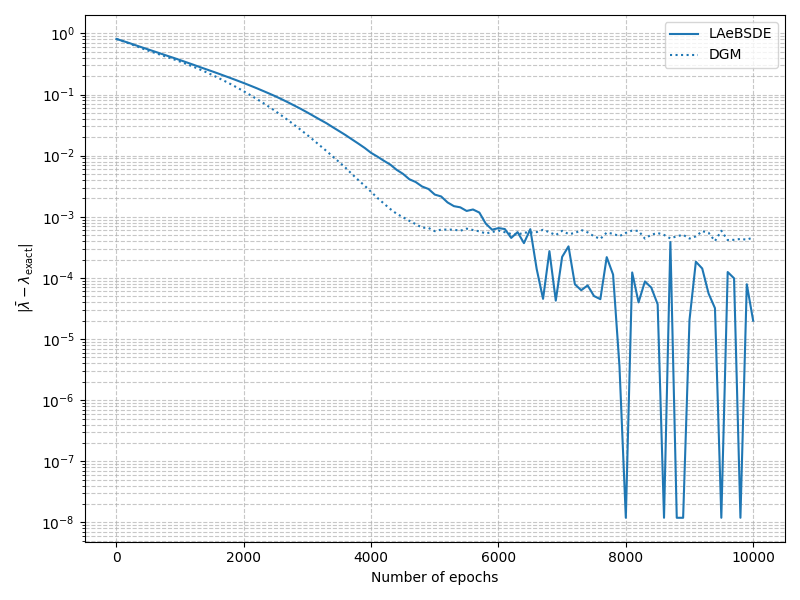}
\caption{Convergence of $|\bar\lambda-\lambda|$ as a function of the number of epochs.}
\label{fig:tanh_lambda_conv}
\end{subfigure}
\caption{Error convergence for the hyperbolic tangent example (H).}
\label{fig}
\end{figure}
The training behavior of both solvers is reported in Figure \ref{fig:tanh_l2_conv}. Both methods show a fast initial decrease of the $L^2(\nu)$-errors, which stabilize around $10^{-2}$ after $4000$ epochs. The LAeBSDE approximation provides the smallest final errors for both $y$ and $z$. Figure \ref{fig:tanh_lambda_conv} illustrates the error convergence of the ergodic cost during training. For the DGM, the ergodic constant error decreases sharply and stabilizes around $10^{-3}$ once the error on $y$ and $z$ have themselves stabilized. This is consistent with the structure of the DGM solver, where $\bar{\lambda}$ is approximated by the invariant-measure representation \eqref{eq:lambda_stat}, whose accuracy is driven by that of $(\bar{y}, \bar{z})$. By contrast, in the LAeBSDE, $\bar{\lambda}$ is a trainable scalar optimized jointly with the networks; its trajectory displays spikes, inherited from the random horizon and from the time-discretization error that enters the loss at each gradient step. Ultimately, it reaches a higher accuracy, of order $10^{-4}$. 

Figure \ref{fig:tanh_yz} compares the approximated solutions $\bar{y}^{i}$ and $z^{i}$ produced by the two solvers against the exact solution, over the main support of the invariant law. Both approaches recover the qualitative structure of the solution, the two regimes crossing near $v = 0$ where the normalization $y^{1}(0) = 1$ is enforced, while the gradients $z^{1}$ and $z^{2}$ are of opposite sign, small, and nearly flat over the relevant range. Overall the two methods provide approximations close to the exact solution on the major part of the invariant distribution and degrades only mildly towards the tails, where samples are sparce.
\begin{figure}[!ht]
\centering
\begin{subfigure}[t]{0.485\textwidth}
\centering
\includegraphics[width=\linewidth]{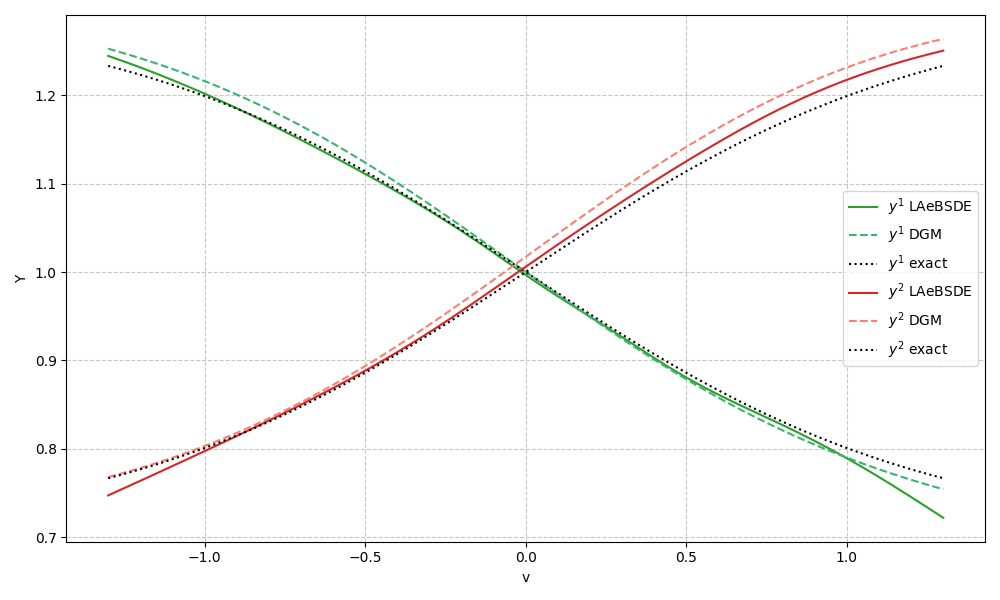}
\caption{Approximated solutions $\bar{y}^{1}, \bar{y}^{2}$ against the exact solution.}
\label{fig:tanh_y}
\end{subfigure}\hspace{0.02\textwidth}
\begin{subfigure}[t]{0.485\textwidth}
\centering
\includegraphics[width=\linewidth]{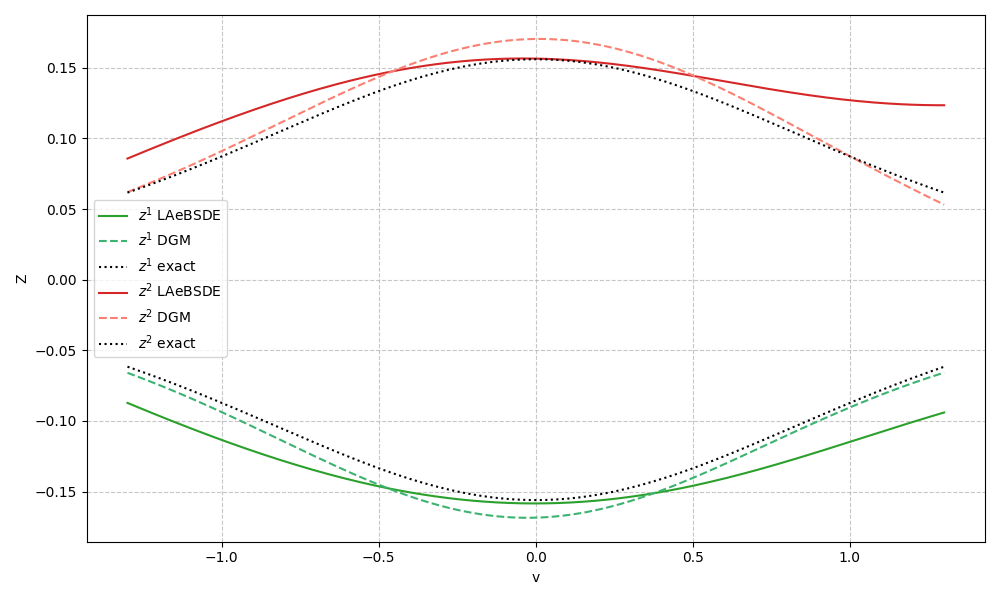}
\caption{Approximated gradient component solution $\bar{z}^{1}, \bar{z}^{2}$ against the exact solution.}
\label{fig:tanh_z}
\end{subfigure}
\caption{Regime-wise solutions $y^{i}, \, z^{i}$ as functions of the factor $v$, for the LAeBSDE and DGM solvers compared with the exact
solution.}
\label{fig:tanh_yz}
\end{figure}

\paragraph{Time discretization parameter for LAeBSDE} The sensitivity of the LAeBSDE scheme with respect to the minimal horizon $T_{0}$ and to the time step $h$ is reported in Tables \ref{tab:tanh_TH} and \ref{tab:tanh_h}.  Increasing $T_{0}$ from $0.1$ to $2.0$ improves all errors, particularly the error on the ergodic constant, which decreases from order $10^{-2}$ to $10^{-7}$ over the tested range. This behavior is coherent with the ergodic nature of the problem: increasing $T_{0}$ lengthens the simulated trajectories, hence enriches the loss with information on the long-run
behavior of the solution. From a numerical point of view, the term $\bar{\lambda} \bar{\tau}$ appearing in the aggregation of error terms \eqref{eq:local_err_LA} acts as a long-time average over the whole trajectory, which makes the loss sensitive to a misspecification of $\bar{\lambda}$. However, this effect saturates for large horizons, see the case $T_{0} = 5$ in Table \ref{tab:tanh_TH}, when the gain in long-run information is balanced by the accumulation of discretization errors. 
\begin{table}[h!]
\centering
\begin{tabular}{c|cccc}
\toprule
$T_{0}$ & $\mathcal E_{L^2(\nu)}(\bar y, M)$ & $\mathcal E_{L^2(\nu)}(\bar z, M)$ & $|\bar\lambda - \lambda|$ & time (s) \\
\midrule
0.10 & $2.25\times 10^{-2}$ & $1.33\times 10^{-2}$ & $1.25\times 10^{-3}$ & $453.9$ \\
0.25 & $1.74\times 10^{-2}$ & $1.03\times 10^{-2}$ & $7.33\times 10^{-3}$ & $460.7$ \\
0.50 & $1.10\times 10^{-2}$ & $9.97\times 10^{-3}$ & $3.11\times 10^{-3}$ & $502.5$ \\
1.00 & $2.96\times 10^{-3}$ & $8.91\times 10^{-3}$ & $1.19\times 10^{-8}$ & $542.7$ \\
2.00 & $4.68\times 10^{-3}$ & $8.42\times 10^{-3}$ & $1.19\times 10^{-8}$ & $648.6$ \\
5.00 & $5.02\times 10^{-2}$ & $4.22\times 10^{-3}$ & $1.19\times 10^{-8}$ & $873.5$ \\
\bottomrule
\end{tabular}
\caption{LAeBSDE errors and training times for different values of the minimal horizon $T_{0}$, with $h=0.01$, in example (T).}
\label{tab:tanh_TH}
\end{table}
As shown in Table \ref{tab:tanh_h}, refining the time step $h$ improves the discretization and yields smaller errors on $\bar y$, $\bar z$, and $\lambda$, but the gain is moderate. This comes at a significant computational cost: decreasing $h$ increases the number of time steps before the random horizon, and therefore substantially increases the training time.
\begin{table}[h!]
\centering
\begin{tabular}{c|cccc}
\toprule
$h$ & $\mathcal E_{L^2(\nu)}(\bar y, M)$ & $\mathcal E_{L^2(\nu)}(\bar z, M)$ & $|\bar\lambda - \lambda|$ & time (s) \\
\midrule
0.005 & $5.95\times 10^{-3}$ & $8.28\times 10^{-3}$ & $1.38\times 10^{-7}$ & 722.65 \\
0.010 & $7.38\times 10^{-3}$ & $7.83\times 10^{-3}$ & $1.38\times 10^{-7}$ & 405.87 \\
0.020 & $4.10\times 10^{-3}$ & $7.70\times 10^{-3}$ & $4.37\times 10^{-4}$ & 298.19 \\
0.050 & $5.25\times 10^{-3}$ & $9.03\times 10^{-3}$ & $1.43\times 10^{-4}$ & 192.79 \\
0.100 & $1.08\times 10^{-2}$ & $1.07\times 10^{-2}$ & $1.51\times 10^{-4}$ & 169.16 \\
\bottomrule
\end{tabular}
\caption{LAeBSDE errors with respect to the time step $h$ in example (T), with $T_{0} = 1$.}
\label{tab:tanh_h}
\end{table}

\paragraph{Switching rate sensitivity} We next evaluate the behavior of the two algorithms as the number of regimes $I$ increases. For each value of $I$, the transition matrix has diagonal entries $-0.8$ and constant off-diagonal entries $q^{ij} = 0.8/(I-1)$, $i\neq j$. The results are gathered in Table \ref{tab:regimes}. Both methods remain accurate, the LAeBSDE yields the smallest errors on $\bar y$ and $\bar z$ for all values of $I$, and these errors stay essentially constant as the number of regimes grows, the error on $\bar y$ remaining of order $10^{-3}$ from $I=2$ to $I=20$. The DGM, in contrast, presents a mild increase of its error on $\bar y$ with the number of regimes, from $8.62\times10^{-3}$ at $I=2$ to $1.14\times10^{-2}$ at $I=20$. This is coherent with the deterioration of approximation quality of deep learning PDE with the dimension of the problem, as observed in \cite{sirignano2018dgm}. The errors on $\bar z$ remain of order $10^{-3}$ for both methods and do not display this trend. The ergodic constant is recovered up to numerical precision by both methods, with an error at the level of round-off in almost all cases.
The main difference lies in the computational cost. The cost of the LAeBSDE grows rapidly with $I$ since each gradient step requires the simulation of a batch of trajectories up to their random horizon and the evaluation of the $I$-dimensional coupling term along each trajectory. The DGM solves the stationary PDE system directly and avoids any time discretization, thus remaining 10 times faster compared to the LAeBSDE.
\begin{table}[h!]
\centering
\begin{tabular}{c|c|l|ccccc}
\toprule
$I$ & $\lambda$ & Method  & $\mathcal E_{L^2(\nu)}(\bar y, M)$ & $\mathcal E_{L^2(\nu)}(\bar z, M)$ & $|\bar\lambda-\lambda^\star|$ & time (s) \\
\midrule
\multirow{2}{*}{$2$}  & \multirow{2}{*}{$8.11\times10^{-1}$} & LAeBSDE & $4.44\times10^{-3}$ & $7.98\times10^{-3}$ & $5.53\times10^{-5}$ & $420.9$  \\
                      &                                      & DGM     & $8.62\times10^{-3}$ & $1.03\times10^{-2}$ & $1.19\times10^{-8}$ & $30.1$   \\
\midrule
\multirow{2}{*}{$5$}  & \multirow{2}{*}{$6.03\times10^{-1}$} & LAeBSDE & $2.85\times10^{-3}$ & $6.05\times10^{-3}$ & $1.19\times10^{-8}$ & $522.9$  \\
                      &                                      & DGM      & $1.10\times10^{-2}$ & $9.08\times10^{-3}$ & $1.19\times10^{-8}$ & $49.4$   \\
\midrule
\multirow{2}{*}{$10$} & \multirow{2}{*}{$5.70\times10^{-1}$} & LAeBSDE  & $3.82\times10^{-3}$ & $5.95\times10^{-3}$ & $1.19\times10^{-8}$ & $1199.3$ \\
                      &                                      & DGM     & $1.11\times10^{-2}$ & $8.10\times10^{-3}$ & $1.19\times10^{-8}$ & $110.9$  \\
\midrule
\multirow{2}{*}{$20$} & \multirow{2}{*}{$5.57\times10^{-1}$} & LAeBSDE & $4.52\times10^{-3}$ & $5.65\times10^{-3}$ & $1.19\times10^{-8}$ & $3421.9$ \\
                      &                                      & DGM    & $1.14\times10^{-2}$ & $8.18\times10^{-3}$ & $1.19\times10^{-8}$ & $261.8$  \\
\bottomrule
\end{tabular}
\caption{Approximation errors across the number of regimes $I$ ($M=10^5$ Monte-Carlo samples).}
\label{tab:regimes}
\end{table}

\subsection{Regime switching power forward utility} \label{sect:4_powerlinear}
We now focus on the systems of ergodic BSDEs associated with regime-switching forward utilities in power form. We consider a market alternating between two economic regimes: a \emph{growth} regime ($i = 1$) with a high risk premium, and a more \emph{conservative} regime ($i = 2$) with a lower one. The single risky asset is driven by a one-dimensional Brownian motion, as in \eqref{eq:stock_price}, with regime-dependent volatility $\sigma^{i} > 0$.  When there is no constraint on the portfolio, that is $\Rr^i = \R^{d}$ for all $i\in \I$, the generator \eqref{eq:powergenswitch} reduces to
\begin{eqnarray} \label{eq:powergen_sect4}
F^{i}(v, z) = \frac{\delta}{2(1 - \delta)} \NRM{z + \theta^{i}(v)}^{2} + \frac{\NRM{z}^{2}}{2}.
\end{eqnarray}
We consider a truncated affine risk premium $\theta^{i}(v) = \varphi_{b}( a_{i} + \theta_{i} v)$ for $i \in \I$,  with $a_{i} \in \R$, $\theta_{i} > 0$ and $b > 0$, where $\varphi_{b}$ denotes the projection on the centered ball of radius $b$. Denoting $\theta_{\max} = \underset{i \in \I}{\max \theta_{i}}$, Assumption \ref{ass:FGgrowthass} holds with $C_{v} = \frac{\ABS{\delta}}{1 - \delta} \max(1, b)\theta_{\max}$. 

We assign a higher market price of risk and a lower stock's volatility to the growth regime, namely
\begin{eqnarray*}
a_{1} = 0.4, \quad a_{2} = -0.1 
\qquad
\theta^{1} = 0.2, \quad \theta^{2} = 0.05,
\qquad
\sigma^{1} = 0.15, \quad \sigma^{2} = 0.3.
\end{eqnarray*}
The remaining parameters are set to $\kappa = 0.8$, $\mu = 1.5$, $T_{0} = 1$, $v_{0} = m = 0$, $\delta = 0.25$, and $b = 1$, and the normalization is fixed to $y^{1}(0) = 1$. The stochastic factor $V$ then admits the Gaussian invariant law $\N(0, 0.213)$, whose $95\%$ interval is approximately $D = \SBRA{-0.91, 0.91}$. Over this range, the growth premium $\theta^{1}(v) \in \SBRA{0.22, 0.58}$ remains positive, with mean $0.4$, and the conservative premium $\theta^{2}(v) \in \SBRA{-0.15, -0.06}$ remains negative with mean $-0.1$. The regime process is the continuous-time Markov chain $\alpha$ with asymmetric transition rate matrix
\begin{eqnarray}
Q = \begin{pmatrix}
-0.3 & 0.3 \\
1.0 & -1.0
\end{pmatrix}
\end{eqnarray}
Its unique invariant distribution is $\bar{\alpha} = \PAR{ \tfrac{10}{13},\ \tfrac{3}{13} }$, the mean holding times are approximately $3.33$ in the growth regime and $1$ in the conservative one. This model that the market spends, on average, more than three times as long in expansion as in stagnation.

Since no closed-form solution is available in this example, we assess the quality of the learned solutions through the residuals of the ergodic PDE \eqref{eq:residualsDGM}. For each regime $i$, define 
\begin{eqnarray}
    \Rr^{i}(\bar{y}, \bar{z}, v) = \Lr \bar{y}^{i}(v) - F^{i}(v, \bar{z}^{i}(v)) + \sum_{i=1}^{I} q^{ij} g(\bar{y}^{j}(v) - \bar{y}^{i}(v)) - \bar{\lambda}.
\end{eqnarray}
 and the mean squared residual under the invariant measure, estimated by Monte Carlo over a sample $(v_{1}, \dots, v_{M})$ from $\nu$,
\begin{eqnarray}
    \E_{PDE}(\bar{y}, \bar{z}, M) = \frac{1}{M} \frac{1}{I} \sum_{m=1}^{M} \sum_{i=1}^{I} \ABS{\Rr^{i}(\bar{y}, \bar{z}, v_{m})}^{2}.
\end{eqnarray}
We also report the normalization error $\E_{\text{norm}} = \ABS{\bar{y}^{i_{0}}(v_{0}) - y_{0}}$.  Figure \ref{fig:losses_power_linear} displays the training losses of both solvers, which decrease and stabilize around $10^{-2}$ for the LAeBSDE and $10^{-5}$ for the DGM, as well as the trajectories of the estimators $\bar\lambda$, which converge to consistent values across the two methods.

\begin{figure}[!ht]
\centering
\begin{subfigure}[t]{0.485\textwidth}
\centering
\includegraphics[width=\linewidth]{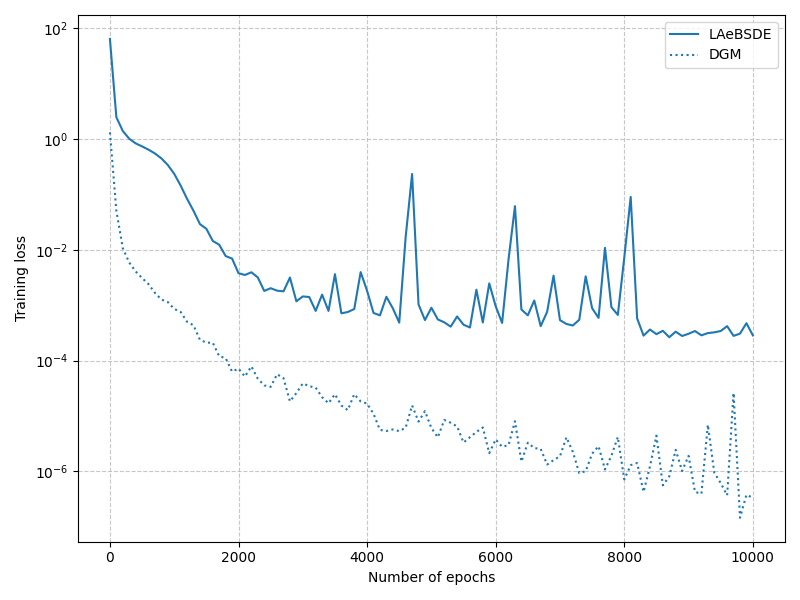}
\caption{Empirical loss functions for the LAeBSDE and DGM during training.}
\label{fig:loss_PL}
\end{subfigure}\hspace{0.02\textwidth}
\begin{subfigure}[t]{0.485\textwidth}
\centering
\includegraphics[width=\linewidth]{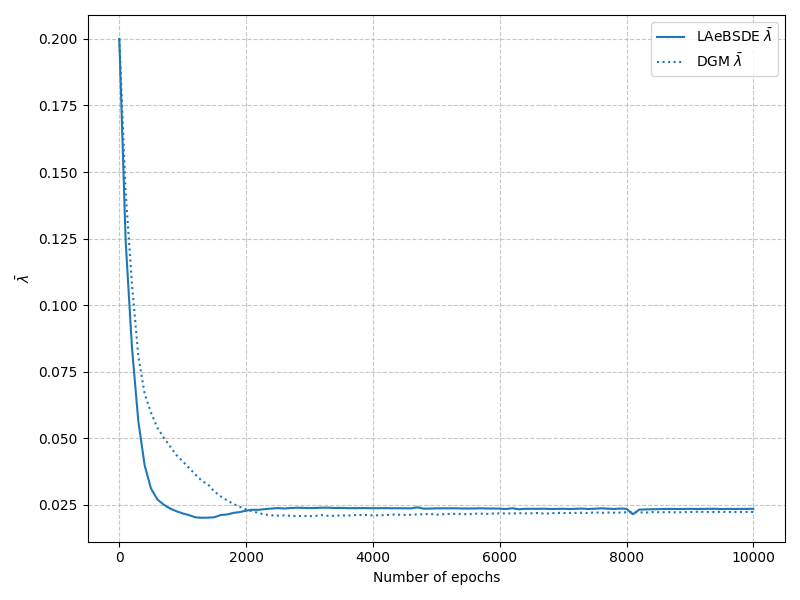}
\caption{Convergence of $\bar{\lambda}$ as a function of the number of epochs.}
\label{fig:loss_lambda_PL}
\end{subfigure}
\caption{Convergence of empirical loss and $\bar{\lambda}$ estimation for both the LAeBSDE and DGM.}
\label{fig:losses_power_linear}
\end{figure}

This offers a good framework for the study of the associated regime-switching forward utility \eqref{eq:Usum} generated by a family of utility random fields in power form as \eqref{eq:powerswitchut}, which takes the form
\begin{eqnarray} \label{eq:exU2switchpower}
U(t, x) = \frac{x^{\delta}}{\delta} e^{y^{\alpha_{t}}(v) - \lambda t}, \quad y^{\alpha_{t}}(v) = \sum_{i \in \I} y^{i}(v) \ind_{\BRA{\alpha_{t}=i}}.
\end{eqnarray}
Simulating the jump times of the two Poisson processes $N^{1, 2}$ and $N^{2, 1}$, and the stochastic factor with its Euler scheme, we are able to simulate the regime-switching forward utility for all time $t \geq 0$. Figure \ref{fig:ut_traj} displays a realization of the regime-switching forward utility together with its non-switching counterparts, the vertical lines marking the jump times of the regime chain. The switching utility is higher than the regime frozen utility in regime $1$, and lower in regime $2$. This provide a trajectory of the regime switching forward utility that sits in between the trajectories of regime frozen preferences, this because of the coupling and the shared ergodic cost solution. 
\begin{figure}[!ht]
\centering
\begin{subfigure}[t]{0.485\textwidth}
\centering
\includegraphics[width=\linewidth]{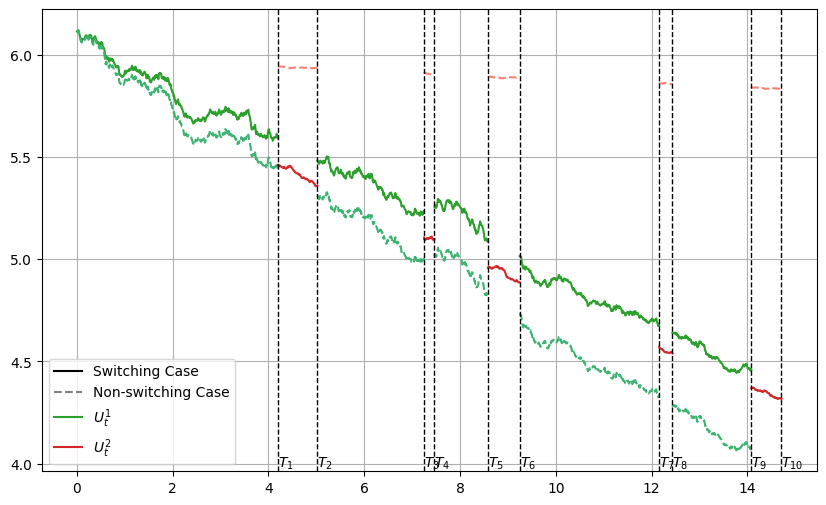}
\caption{Trajectory of the regime switching forward utility \eqref{eq:exU2switchpower}}
\label{fig:ut_traj}
\end{subfigure}\hspace{0.02\textwidth}
\begin{subfigure}[t]{0.485\textwidth}
\centering
\includegraphics[width=\linewidth]{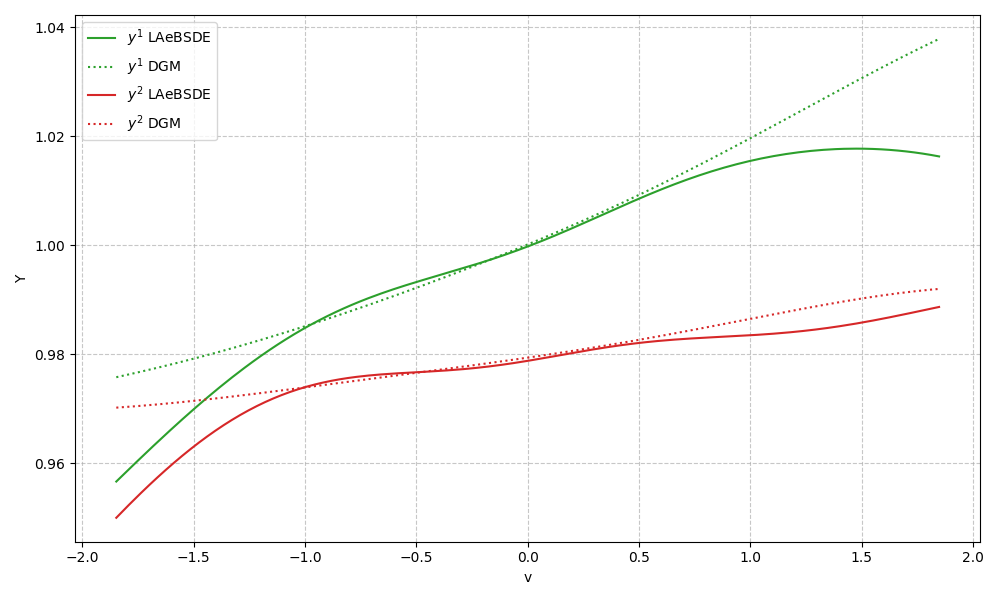}
\caption{Numerical approximation $\bar{y}$ for LAeBSDE and DGM.}
\label{fig:y_fct_PL}
\end{subfigure}
\caption{Trajectory of the regime-switching forward utility and solution}
\label{fig:ut_y_PL}
\end{figure}
Figure \ref{fig:y_fct_PL} displays the solution $\bar{y}$ in each state as a function of $v$. The approximations produced by the two methods coincide on the high density region of the invariant
law's support, with discrepancies growing as $v$ approaches $1$, where samples are sparce.

\paragraph{Impact of regime-switching frequency} We investigate how the intensity of the regime switches affects the solution and its numerical approximation. We scale the transition rate matrix $Q$ by a factor $q \in \BRA{0.01, 0.1, 1, 10, 100}$, that is $Q(q) = qQ$, which leaves the stationary distribution $\alpha$ unchanged while accelerating the chain. The corresponding PDE residuals and normalization errors are reported in Table \ref{tab:pde_residual_Q}.
\begin{table}[h!]
\centering
\begin{tabular}{c|l|cccccc}
\toprule
$q$ ($Q=q\,Q_0$) & Method & $\mathcal E_{L^2(\nu)}(\bar y, M)$ & $\mathcal E_{\mathrm{norm}}$ & $\bar\lambda$ & $\underset{v\in D}{\max}|y^1-y^2|$ & $C_Y$ \\
\midrule
\multirow{2}{*}{$0.01$} & LAeBSDE & $7.79\times10^{-4}$ & $4.33\times10^{-4}$ & $-1.15\times10^{-2}$ & $2.29\times10^{0}$ & \multirow{2}{*}{$1.22\times10^{2}$} \\
                        & DGM     & $8.10\times10^{-4}$ & $6.64\times10^{-4}$ & $-1.13\times10^{-2}$   & $2.25\times10^{0}$ & \\
\midrule
\multirow{2}{*}{$0.1$}  & LAeBSDE & $7.92\times10^{-4}$ & $2.34\times10^{-4}$ & $-2.72\times10^{-2}$ & $3.13\times10^{-1}$ & \multirow{2}{*}{$1.22\times10^{1}$} \\
                        & DGM     & $8.15\times10^{-4}$ & $4.69\times10^{-4}$ & $-2.66\times10^{-2}$   & $3.09\times10^{-1}$ & \\
\midrule
\multirow{2}{*}{$1$}    & LAeBSDE & $8.13\times10^{-4}$ & $7.69\times10^{-4}$ & $-2.94\times10^{-2}$ & $4.76\times10^{-2}$ & \multirow{2}{*}{$1.22\times10^{0}$} \\
                        & DGM     & $8.47\times10^{-4}$ & $4.67\times10^{-5}$ & $-2.86\times10^{-2}$   & $4.80\times10^{-2}$ & \\
\midrule
\multirow{2}{*}{$10$}   & LAeBSDE & $1.24\times10^{-3}$ & $3.99\times10^{-2}$ & $-3.09\times10^{-2}$ & $7.63\times10^{-3}$ & \multirow{2}{*}{$1.22\times10^{-1}$} \\
                        & DGM     & $8.59\times10^{-4}$ & $8.71\times10^{-4}$ & $-2.91\times10^{-2}$  & $5.55\times10^{-3}$ & \\
\midrule
\multirow{2}{*}{$100$}  & LAeBSDE & $3.65\times10^{-3}$ & $9.79\times10^{-1}$ & $-2.88\times10^{-2}$ & $1.64\times10^{-3}$ & \multirow{2}{*}{$1.22\times10^{-2}$} \\
                        & DGM     & $1.70\times10^{-3}$ & $2.41\times10^{-3}$ & $-5.66\times10^{-3}$   & $1.02\times10^{-3}$ & \\
\bottomrule
\end{tabular}
\caption{PDE residuals and normalization errors for different switching-rate scales $q$, and $\delta = -1$ ($M=10^5$ Monte-Carlo samples).}
\label{tab:pde_residual_Q}
\end{table}
\noindent
We observe that the maximal distance between the two regime solutions $\underset{v\in D}{\max}|y^1-y^2|$ decreases by a factor $10$ with $q$, which is coherent with the theoretical coupling bound $C_{Y}$, which scales as $q^{-1}$. In other words, when the chain switches rapidly, the regime specific solutions $y^{1}$ and $y^{2}$ collapse onto one common corrector, whereas for slow switching $(q=0.01)$, the two regimes are nearly decoupled and the solutions differ. The numerical error grows as $q$ increases for both method, and the normalization error of the LAeBSDE deteriorates at high switching rates, since the switching term has a high weight.

\paragraph{Risk aversion sensitivity}
Table~\ref{tab:delta} gathers the errors and estimated ergodic costs for several values of the risk-aversion parameter $\delta$, from the risk-tolerant case $\delta = 0.5$ to the strongly risk-averse case $\delta = -5$. Several observations stand out. First, the residuals deteriorate as $|\delta/(1-\delta)|$ grows: the $\mathcal E_{\mathrm{PDE}}$ of both solvers increases by a factor $10$ between $\delta=0.25$ and $\delta=0.5$. This is expected, as the factor $\frac{\delta} {2(1-\delta)}$ governs both the size of the quadratic nonlinearity in \eqref{eq:powergen_sect4} and the growth constant $C_{v}$, and hence the magnitude of $z$. The two methods are comparable on the PDE residual, and both keep the normalization error below $10^{-3}$.
\begin{table}[h!]
\centering
\begin{tabular}{c|l|cccc}
\toprule
$\delta$ & Method & $\mathcal E_{\mathrm{PDE}}$ & $\mathcal E_{\mathrm{norm}}$ & $\bar\lambda$ & time (s) \\
\midrule
$0.5$ & LAeBSDE & $5.32\times 10^{-3}$ & $1.30\times 10^{-4}$ & $8.02\times 10^{-2}$ & $331.8$ \\
      & DGM     & $5.61\times 10^{-3}$ & $3.08\times 10^{-5}$ & $7.89\times 10^{-2}$ & $8.5$ \\
\midrule
$0.25$ & LAeBSDE & $4.95\times 10^{-4}$ & $2.28\times 10^{-4}$ & $2.35\times 10^{-2}$ & $329.3$ \\
       & DGM     & $4.91\times 10^{-4}$ & $1.44\times 10^{-4}$ & $2.23\times 10^{-2}$ & $8.5$ \\
\midrule
$-1.0$ & LAeBSDE & $8.13\times 10^{-4}$ & $4.10\times 10^{-5}$ & $-2.93\times 10^{-2}$ & $331.4$ \\
       & DGM     & $8.39\times 10^{-4}$ & $1.81\times 10^{-4}$ & $-2.84\times 10^{-2}$ & $8.7$ \\
\midrule
$-2.0$ & LAeBSDE & $1.42\times 10^{-3}$ & $3.12\times 10^{-3}$ & $-3.84\times 10^{-2}$ & $331.0$ \\
       & DGM     & $1.41\times 10^{-3}$ & $1.21\times 10^{-4}$ & $-3.67\times 10^{-2}$ & $8.7$ \\
\midrule
$-5.0$ & LAeBSDE & $2.11\times 10^{-3}$ & $4.54\times 10^{-4}$ & $-4.59\times 10^{-2}$ & $330.4$ \\
       & DGM     & $2.09\times 10^{-3}$ & $6.25\times 10^{-4}$ & $-4.48\times 10^{-2}$ & $8.5$ \\
\bottomrule
\end{tabular}
\caption{PDE residuals $\mathcal E_{\mathrm{PDE}}$, normalization errors and estimated
ergodic costs for the LAeBSDE and DGM solvers, for several values of the risk-aversion
parameter $\delta$.}
\label{tab:delta}
\end{table}
Additionally, the estimated ergodic cost $\bar\lambda$ reproduces across all runs the sign relation $\operatorname{sign}\bar\lambda = \operatorname{sign}\delta$. The two solvers agree on $\bar\lambda$ to within $5\%$ in every case, indicating a coherent approximation form both methods. The relation follows by evaluating the criterion at the admissible constant strategy $\pi \equiv 0$ between $0$ and $t$, and exploiting the supermartingale property of the preference criterion $U$, which reads
\begin{eqnarray} \label{eq:supermart_slambda}
    \frac{x_{0}^{\delta}}{\delta} e^{-\lambda t} \e \SBRA{e^{y^{\alpha_{t}}(V_{t})}} \leq \frac{x_{0}^{\delta}}{\delta}\,e^{y^{\alpha_{0}}(V_{0})}.
\end{eqnarray}
Recall that $y^{\alpha_{t}}$ is bounded by sub-linearity of $y^{i}$ and Proposition 5 in \cite{hu2019ergodic}. Dividing by $\frac{x_{0}^{\delta}}{\delta}$ and taking the logarithm in \eqref{eq:supermart_slambda} yields
\begin{itemize}
    \item if $\delta > 0$, $\frac{1}{t} \PAR{\ln \e\SBRA{e^{y^{\alpha_{t}}(V_{t})}} - y^{\alpha_{0}}(V_{0})} \leq \lambda$, so that letting $t \to \infty$ leads $0 \leq \lambda$. 
    \item if $\delta < 0$, $\frac{1}{t} \PAR{\ln \e\SBRA{e^{y^{\alpha_{t}}(V_{t})}} - y^{\alpha_{0}}(V_{0})} \geq \lambda$, so that letting $t \to \infty$ leads $0 \geq \lambda$.
\end{itemize}

\paragraph{Optimal strategies and comparison with the single-regime case.}
Recall that $\pi^{i,*}(v) = \frac{\bar z^i(v) + \theta^i(v)}{1-\delta}$ denotes the optimal strategy rescaled by the stock volatility, see \eqref{wealth_switch}. Figure \ref{fig:strategies} displays the effective proportion of wealth invested in stock $i$, that is $\bar\pi^{i,*}(v) = \frac{1}{\sigma_i}\, \pi^{i,*}(v)$, obtained from the DGM. We report it both in the coupled (switching) case and in the decoupled case, where each regime is treated as a stand-alone market and the corresponding scalar ergodic BSDE of \cite{liang2017representation} is solved with the same parameters. For $\delta = -1$, the switching strategy is very close to the decoupled one and leads to realistic allocation levels. For $\delta = 0.5$, the allocations are heavily leveraged and the discrepancy between the coupled and decoupled strategies becomes visible.

This behavior can be explained from the structure of the solution. The regime coupling affects the strategy \emph{only} through the $\bar z^i$ component. The size of $z^{i}$ depends on the weight $\frac{\delta}{2(1-\delta)}$ in the generator \eqref{eq:powergen_sect4}, which tends to $0$ as $\delta \to 0^-$ and remains bounded by $1/2$ as $\delta \to -\infty$, whereas it blows up as $\delta \to 1^-$. This weight also rescales the allocation size, leading to leveraged position for $\delta$ close to $1$. Consequently, for extreme risk averse agents ($\delta$ negative), the optimal allocation is dominated by the myopic component, and the anticipation of regime switches encoded in $z^{i}$ is minimal. For risk-tolerant agents ($\delta$ positive, close to $1$), the component $z$ is amplified, and the switching effect becomes visible, at the cost of extreme leverage. Figure \ref{fig:strategies} illustrate this observation.

\begin{figure}[!ht]
\centering
\begin{subfigure}[t]{0.485\textwidth}
\centering
\includegraphics[width=\linewidth]{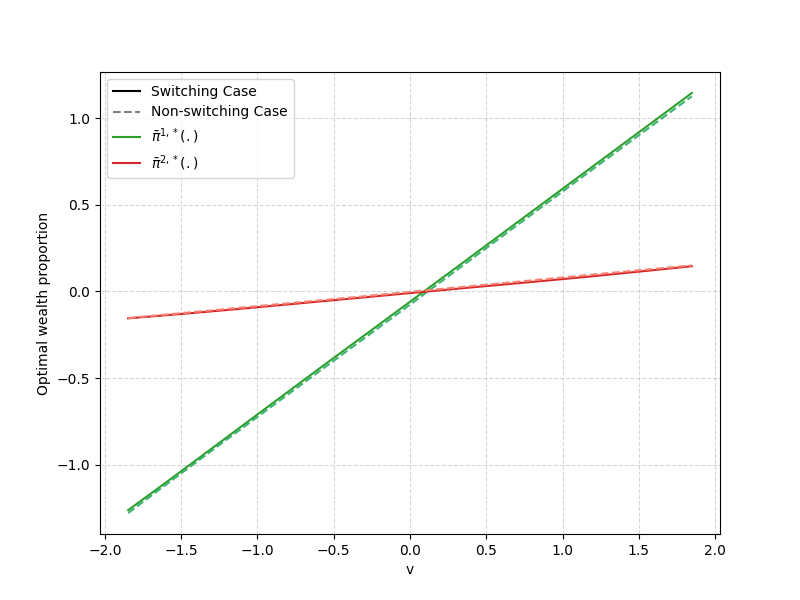}
\caption{Optimal allocation for $\delta = -1$.}
\label{fig:strat_dm1}
\end{subfigure}\hspace{0.02\textwidth}
\begin{subfigure}[t]{0.485\textwidth}
\centering
\includegraphics[width=\linewidth]{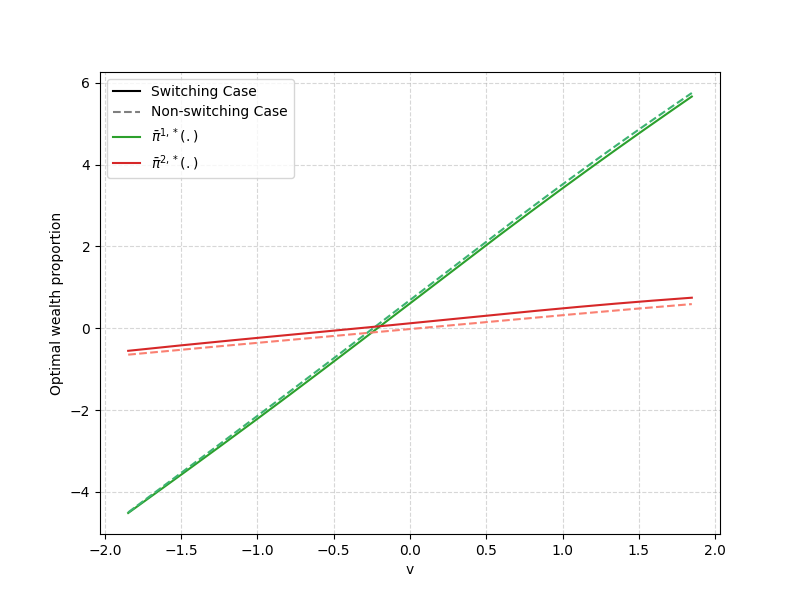}
\caption{Optimal allocation for $\delta = 0.5$}
\label{fig:strat_dp05}
\end{subfigure}
\caption{Optimal portfolio allocation for different risk aversion parameters.}
\label{fig:strategies}
\end{figure}

\paragraph{Interpretation of the ergodic cost.}
For $\delta = 0.5$, the ergodic cost is estimated at $\bar\lambda = 8.02 \times 10^{-2}$. In the representation $U^i(t,x) = \frac{x^\delta}{\delta} e^{y^i(V_t) - \lambda t}$, the constant $\lambda$ plays the role of an endogenous \emph{time-preference rate}: a positive $\lambda$ exponentially discounts future utility and thus reflects a preference for the present, whereas a negative $\lambda$ up-weights future utility. The sign of this rate is tied to the agent's risk preference: risk tolerant agents ($\delta > 0$) discount future utility as a positive rate, while risk averse agents ($\delta < 0$) exhibit a negative rate that favors future utility.

Across the whole range of risk aversions considered, the ergodic cost remains of order $10^{-2}$ in absolute value, which is consistent with the magnitudes considered for long-term social discount rates in the policy literature. The magnitude of such rates is critical in long-term decision problems, and has been widely debated, in particular in environmental policy and sustainable development models, where excessively high discount rates are criticized for undervaluing the welfare of future generations; see \cite{lecoq2004}, \cite{stern2007economics}. A noticeable feature of forward utilities derived from systems of ergodic BSDEs is that this rate is \emph{not} an exogenous input of the model but emerges as part of the solution $(y, z, \lambda)$, jointly determined by the market dynamics, namely the risk premia and the factor's ergodicity, and by the switching mechanism. This opens the possibility of calibrating the model parameters so as to reflect a desired time-preference structure.

\newpage 

\appendix

\section{Spaces for regularity of semimartingales} \label{annex:spaces}

Let $\U_{std}$ denote the set of standard deterministic utility functions. We next introduce spaces for studying the regularity of semimartingales in terms of their local characteristics, following \cite{kunita1997stochastic}, \cite{ikeda2014stochastic}.
\paragraph{Definition of seminorms -} Let $\beta$ be an $\R^{k}$-valued random field of class \hfill \linebreak
$C^{m, \delta}\PAR{]0, +\infty[}$, with $m$ a nonnegative integer and $\delta$ a number in $(0, 1]$, i.e. $\beta$ is $m$ times differentiable in $x$ and its $m^{\text{th}}$ derivative is $\delta$-Hölder, for any $t$, almost surely. We introduce the following family of random Hölder $K$-seminorms aiming to control asymptotic behavior of $\beta$ and the regularity of its Hölder derivatives, for any $K \subset ]0, +\infty[$
\begin{eqnarray*}
\NRM{\beta}_{m, K}(t, \omega) &=& \underset{x \in K}{\sup} \frac{\NRM{\beta(t, x, \omega)}}{x} + \sum_{1 \leq j \leq m} \underset{x \in K}{\sup} \NRM{\partial_{x}^{j} \beta(t, x, \omega)} \\
\NRM{\beta}_{m, \delta, K}(t, \omega) &=& \NRM{\beta}_{m, K}(t, \omega) + \underset{x, y \in K}{\sup} \frac{\NRM{\partial_{x}^{m} \beta(t, x, \omega) - \partial_{x}^{m} \beta(t, y, \omega)}}{\ABS{x-y}^{\delta}}.
\end{eqnarray*}
As mentioned in \cite{nicole2013exact}, the first term of these random semi-norms is divided by $x$, in order to control behavior in the neighborhood of $x=0$, by imposing at most linear vanishing at the boundary. This normalization also preserves traditional asymptotic results from \cite{kunita1997stochastic}.

\paragraph{Associated function spaces -} The previous norms are related to the space parameter. We add the temporal dimension by requiring these seminorms (or their square) to be integrable in time with respect to Lebesgue measure on $[0, T]$. We then define the following sets:
\begin{enumerate}
    \item $\K^{m}_{\loc}$ (resp. $\overline{\K}^{m}_{\loc}$) denotes the set of $C^{m}$-random fields $\beta$ such that $\frac{\beta}{x}$ and $\partial_{x}^{k} \beta$ for $k \leq m$ are $\Lr^{1}$ (resp. $\Lr^{2}$)-locally bounded, that is for any compact $K \subset ]0, +\infty[$ and any $T$, 
    \begin{eqnarray*}
    \int_{0}^{T} \NRM{\beta}_{m, K}(t, \omega)dt < \infty, \quad \PAR{\text{resp. } \int_{0}^{T} \NRM{\beta}_{m, K}^{2}(t, \omega)dt < \infty}.
    \end{eqnarray*}
    \item $\K^{m, \delta}_{\loc}$ (resp. $\overline{\K}^{m, \delta}_{\loc}$) denotes the set of $C^{m, \delta}$-random fields such that for any compact $K \subset ]0, +\infty[$ and any $T$, 
    \begin{eqnarray*}
    \int_{0}^{T} \NRM{\beta}_{m, \delta, K}(t, \omega)dt < \infty, \quad \PAR{\text{resp. } \int_{0}^{T} \NRM{\beta}_{m, \delta, K}^{2}(t, \omega)dt < \infty}.
    \end{eqnarray*}
    \item When these norms are defined on the whole space $]0, +\infty[$, the derivatives up to a certain order are bounded in the spatial parameter, with an integrable (resp. square integrable) random bound, so that we use the notation $\K_{b}^{m}, \, \overline{\K}_{b}^{m}$ or $\K_{b}^{\delta, m}, \, \overline{\K}_{b}^{m, \delta}$.
\end{enumerate}

\newpage 
\section{Proof of Lemma \ref{lem:runningsupV}}
\label{proofLemma 1.1}
For the sake of completeness, we provide the proof of the estimate of the running supremum in Lemma \ref{lem:runningsupV}.

\begin{proof}
Applying Itô's formula to $\ABS{V_{t}}^{2}$, and using the dissipative Assumption \ref{ass:weakdissass} yields
\begin{eqnarray*}
d\ABS{V_{t}}^{2} &=& \PAR{2V_{t} \mu(V_{t}) + \ABS{\kappa}^{2}}dt + 2V_{t} \kappa^{\top}dW_{t} \\
&\leq& \PAR{-C_{\mu} \ABS{V_{t}}^{2} + \frac{\ABS{\mu(0)}^{2}}{C_{\mu}} + \ABS{\kappa}^{2}}dt + 2V_{t} \kappa^{\top}dW_{t}.
\end{eqnarray*}
Dropping the negative term and taking the supremum over $s \in \SBRA{0, \tau}$,
\begin{eqnarray*}
\underset{0 \leq t \leq \tau}{\sup} \ABS{V_{t}}^{2} \leq \ABS{v_{0}}^{2} + C_{0} \tau + 2 \underset{0 \leq t \leq \tau}{\sup} \ABS{\int_{0}^{t} V_{s}\kappa^{\top} dW_{s}}.
\end{eqnarray*}
We now aim to obtain a bound of the above quantity in $L^{q}$. Elevating to the power $q$ and taking expectation, we get that there exists a constant $C_{q}$ such that
\begin{eqnarray*}
\e \SBRA{\underset{0 \leq t \leq \tau}{\sup} \ABS{V_{t}}^{2q}} \leq C_{q} \PAR{\ABS{v_{0}}^{2q} + \e \SBRA{\tau^{q}} + \e \SBRA{\underset{0 \leq t \leq \tau}{\sup} \ABS{\int_{0}^{t} V_{s}\kappa^{\top} dW_{s}}^{q}}}.
\end{eqnarray*}
An application of the Burkholder-Davis-Gundy inequality in $L^{q}$, then allow to bound the local martingale, leading
\begin{eqnarray}
\e \SBRA{\underset{0 \leq t \leq \tau}{\sup} \ABS{V_{t}}^{2q}} &\leq& C_{q} \PAR{\ABS{v_{0}}^{2q} + \e \SBRA{\tau^{q}} + C_{q} \ABS{\kappa}^{q} \e \SBRA{\PAR{\int_{0}^{\tau} \ABS{V_{s}}^{2}ds}^{q/2}}} \nonumber \\
&\leq& C_{q} \PAR{\ABS{v_{0}}^{2q} + \e \SBRA{\tau^{q}} + C_{q} \ABS{\kappa}^{q} \e \SBRA{\tau^{\frac{q}{2}} \underset{0 \leq t \leq \tau}{\sup} \ABS{V_{t}}^{q}}}. \label{eq:BDGrunningV_p2}
\end{eqnarray}
Next applying Young's inequality, for any $\epsilon > 0$, there exists $C_{\epsilon}$ such that
\begin{eqnarray*}
\tau^{\frac{q}{2}} \underset{0 \leq t \leq \tau}{\sup} \ABS{V_{t}}^{q} \leq \epsilon \underset{0 \leq t \leq \tau}{\sup} \ABS{V_{t}}^{2q} + C_{\epsilon} \tau^{q}.
\end{eqnarray*}
Injecting this back into \eqref{eq:BDGrunningV_p2} yields 
\begin{eqnarray*}
\e \SBRA{\underset{0 \leq t \leq \tau}{\sup} \ABS{V_{t}}^{2q}} &\leq& C_{q} \PAR{\ABS{v_{0}}^{2q} + \e \SBRA{\tau^{q}} + C_{q} \ABS{\kappa}^{q} ( \epsilon \e \SBRA{\underset{0 \leq t \leq \tau}{\sup} \ABS{V_{t}}^{2q}} + C_{\epsilon} \e \SBRA{\tau^{q}}}.
\end{eqnarray*}
For $\epsilon$ small enough, the term $\e \SBRA{\underset{0 \leq t \leq \tau}{\sup} \ABS{V_{t}}^{2q}}$ can be absorbed on the left hand side, leading 
\begin{eqnarray}
\e \SBRA{\underset{0 \leq t \leq \tau}{\sup} \ABS{V_{t}}^{2q}} \leq C_{q} \PAR{1 + \ABS{v_{0}}^{2q} + \e \SBRA{\tau^{q}}}.
\end{eqnarray}
\end{proof}

\newpage
\section{Proof of Proposition \ref{prop:consistpower}}

\begin{proof} 
Let $i \in \I$ and $U^i$ utility random field of power type verifying \eqref{eq:powUi}, so that $U^i$ have the local characteristics $\beta^i(t,x) = b_t^i U^i(t,x)$
and $\gamma^i(t,x) = \nu_t^i U^i(t,x)$. 
We start this proof by recalling useful relations for power utilities defined by \eqref{eq:powUi}:
    \begin{eqnarray*}
       &  xU^i_x(t,x) = \delta^i U^i(t,x),\\
       & x^2 U_{xx}^i(t,x) = \delta^i(\delta^i - 1) U^i(t,x), \\
       & x  U_{xx}^i(t,x) = (\delta^i - 1) U^i(t,x). 
    \end{eqnarray*}
   Combining this with \eqref{eq:pi_regime_i}  yields that the  optimal policy in regime $i$ is given by: 
   \begin{equation*}
       \pi^{i,*}_t = \Proj_{\Rr^{i}} \left(  \frac{\nu^i_t + \theta^i(V_t)}{1-\delta_i} \right). 
   \end{equation*}
Then, by Theorem \ref{thm:consistswitch}, we have
\begin{equation*}
        Q^i (t,x,\pi^{i,*}_t)= - \NRM{x \pi^{i,*}_t}^{2} = - x^2 \dist^{2} \PAR{\Rr^i, \frac{  \nu^i_t + \theta^{i}(V_t)}{1 - \delta_i}}, 
\end{equation*}
and the time consistency assumption \eqref{eq:driftcondswitch} becomes 
\begin{align*}
    b_t^i U^i(t,x) &  = \frac{1}{2} U^i_{xx}(t,x)x^2 \Proj^{2} \PAR{\Rr^{i}, \frac{  \nu^i_t + \theta^{i}(V_t)}{1 - \delta_i}} - \sum_{j\in \I}\left( U^j(t,x) - U^i(t,x)\right) q^{ij}\\
    & = U^i(t,x) \left( \frac{\delta_i(\delta_i - 1)}{2} \Proj^{2} \PAR{\Rr^{i}, \frac{  \nu^i_t + \theta^{i}(V_t)}{1 - \delta^i}} - \sum_{j\in \I}\left( \frac{U^j(t,x)}{U^i(t,x)} - 1 \right) q^{ij}\right). 
\end{align*}
The time consistency assumption can thus be rewritten as:
\begin{equation*}
      b_t^i  = \frac{\delta_i(\delta_i - 1)}{2} \Proj^{2} \PAR{\Rr^{i}, \frac{  \nu^i_t + \theta^{i}(V_t)}{1 - \delta_i}} - \sum_{j\in \I}\frac{P^j_t}{P^i_t} \frac{\delta_i}{\delta_j} x^{\delta_i - \delta_j} q^{ij}. 
\end{equation*}
By assumption \ref{asslispch_switch},  for $j \neq i $ $q^{ij}>0$. Thus, we have necessarily $\delta_i = \delta_j$ for all $i,j \in \I$, which concludes the proof. 
\end{proof}
\newpage
\bibliographystyle{alpha-fr}
\bibliography{Biblio}

\end{document}